\documentclass[11pt,reqno]{amsart}
\usepackage{amsmath, amssymb, amsthm}
\usepackage{tensor}
\usepackage{enumerate}
\usepackage[colorlinks=true,allcolors=blue]{hyperref}
\usepackage{blkarray}
\usepackage{tensor}

\usepackage{natbib}

\usepackage{tikz}
\usetikzlibrary{cd,arrows,calc}

\newcommand{\C}{\mathbb{C}}
\newcommand{\R}{\mathbb{R}}
\newcommand{\Z}{\mathbb{Z}}
\newcommand{\N}{\mathbb{N}}
\newcommand{\Q}{\mathbb{Q}}

\newcommand{\bK}{\mathbb{K}}
\newcommand{\bT}{\mathbb{T}}

\newcommand{\liet}{\mathfrak{t}}
\newcommand{\lieg}{\mathfrak{g}}

\newcommand{\RQ}{R_\Q}
\newcommand{\RFQ}{R_{F,\Q}}

\newcommand{\cR}{\mathcal{R}}
\newcommand{\cRQ}{\mathcal{R}_{\mathbb{Q}}}
\newcommand{\sgn}[1]{#1^{\text{sgn}}}

\newcommand{\Cuntz}[1]{\mathcal{O}_{#1}}

\newcommand{\Aut}[1]{\text{\normalfont Aut}(#1)}
\newcommand{\id}[1]{\operatorname{id}_{#1}}

\newcommand{\EV}[1]{\text{EV}(#1)}
\newcommand{\cA}{\mathcal{A}}

\newcommand{\cG}{\mathcal{G}}

\newcommand{\cV}{\mathcal{V}}
\newcommand{\sfX}{\mathsf{X}}

\newcommand{\cE}{\mathcal{E}} 
\newcommand{\cL}{\mathcal{L}}

\newcommand{\cF}{\mathcal{F}} 

\newcommand{\Eig}[2]{\text{Eig}(#1,#2)}
\newcommand{\Endo}[1]{\operatorname{End}\left(#1\right)}
\newcommand{\op}{\text{op}}

\newcommand{\lscal}[3]{\tensor*[_{#1}]{\langle #2, #3 \rangle}{}}
\newcommand{\rscal}[3]{\tensor*{\langle #1, #2 \rangle}{_{#3}}}

\newcommand{\rcpt}[2]{\mathcal{K}_{#1}\left(#2\right)}

\newcommand{\rbdd}[2]{\mathcal{L}_{#1}\left(#2\right)}

\newcommand{\Viso}{\mathcal{V}^{{\rm iso}}_{\C}}
\newcommand{\Vgr}{\mathcal{V}^{{\rm gr}}_{\C}}

\newcommand{\MF}{\mathsf{M}_\mathnormal{F}^{\infty}}
\newcommand{\MFfac}[1]{\mathsf{M}_\mathnormal{F,#1}^{\infty}}

\newcommand{\Waff}{W^{\text{\normalfont aff}}}

\DeclareMathOperator{\Ad}{Ad}

\newcommand{\Udet}[1]{U_{\tau}(#1)}

\newcommand{\GLres}[1]{GL_{\rm res}(#1)}
\newcommand{\Ures}[1]{U_{\rm res}(#1)}
\newcommand{\Uidres}[1]{U^0_{\rm res}(#1)}
\newcommand{\GLidres}[1]{GL^0_{\rm res}(#1)}

\newtheorem{theorem}{Theorem}[section]
\newtheorem{lemma}[theorem]{Lemma}
\newtheorem{corollary}[theorem]{Corollary}
\newtheorem{prop}[theorem]{Proposition}

\theoremstyle{definition}
\newtheorem{definition}[theorem]{Definition}
\newtheorem{remark}[theorem]{Remark}
\newtheorem{example}[theorem]{Example}

\makeatletter
\newcommand{\extp}{\@ifnextchar^\@extp{\@extp^{\,}}}
\def\@extp^#1{\mathop{\bigwedge\nolimits^{\!#1}}}
\makeatother

\begin{document}
\title[Spectral Sequence Computation of Higher Twisted $K$-Groups]{Spectral Sequence Computation of \\ Higher Twisted $K$-Groups of $ SU(n)$}
\author{David E.\ Evans \and Ulrich Pennig}
\address{Cardiff University, School of Mathematics, Senghennydd Road, Cardiff, CF24 4AG, Wales, UK}
\email{EvansDE@cardiff.ac.uk}
\email{PennigU@cardiff.ac.uk}


\begin{abstract}
Motivated by the Freed-Hopkins-Teleman theorem we study graded equivariant higher twists of $K$-theory for the groups $G = SU(n)$ induced by exponential functors. We compute the rationalisation of these groups for all $n$ and all non-trivial functors. Classical twists use the determinant functor and yield equivariant bundles of compact operators that are classified by Dixmier-Douady theory. Their equivariant $K$-theory reproduces the Verlinde ring of conformal field theory. Higher twists give equivariant bundles of stable UHF algebras, which can be classified using stable homotopy theory. Rationally, only the $K$-theory in degree $\dim(G)$ is again non-trivial. The non-vanishing group is a quotient of a localisation of the representation ring $R(G) \otimes \Q$ by a higher fusion ideal~$J_{F,\Q}$. We give generators for this ideal and prove that these can be obtained as derivatives of a potential. For the exterior algebra functor, which is exponential, we show that the determinant bundle over $LSU(n)$ has a non-commutative counterpart where the fibre is the unitary group of the UHF algebra.
\end{abstract}

\maketitle

\section{Introduction}

\subsection{History and Motivation}
$K$-theory has its roots in Grothendieck's generalisation of the Riemann-Roch theorem in the 1950s, which he formulated in the language of algebraic geometry and coherent sheaves \cite{paper:GrothendieckRiemannRoch}. Building on these foundational insights, Atiyah and Hirzebruch recognised that similar principles could be applied in a purely topological setting, leading them to develop topological $K$-theory \cite{paper:AtiyahHirzebruch,book:Atiyah}. For a compact Hausdorff space $X$ the group $K^0(X)$ is the group completion of the monoid obtained from the isomorphism classes of complex, finite-dimensional vector bundles over~$X$ under the direct sum. What transforms this algebraic construction into a powerful topological invariant is its extension to a cohomology theory $X \mapsto K^*(X)$ with values in $\Z$-graded abelian groups. Bott periodicity provides natural isomorphisms $K^i(X) \cong K^{i+2}(X)$ that dramatically simplify the theory by reducing all computations to just two degrees. It turns the long exact sequence of pairs $(X,Y)$ with $Y \subseteq X$ involving the relative $K$-groups $K^*(X,Y)$ into the characteristic six-term exact sequence that makes $K$-theory computations tractable. 

The importance of topological $K$-theory was highlighted by the foundational work of Atiyah and Singer on index theory \cite{paper:AtiyahSingerI, paper:FreedAS}: principal symbols of elliptic differential operators naturally define classes in $K^0(T^*M)$, the index of such an operator can be computed through a pairing between $K$-theory and its dual $K$-homology, and families of operators parametrised by a topological space $X$ have indices living in $K^0(X)$. Applications in physics include quantum fields as section of vector bundles and $D$-brane charges in string theory \cite{paper:WittenDBranes}.

By now there are many variations on the theme of $K$-theory. The most relevant one for this paper is the operator-algebraic $K$-theory of $C^*$-algebras. For a unital $C^*$-algebra $A$ the group $K_0(A)$ is defined as the group completion of isomorphism classes in a suitable category of projective $A$-modules. The functor $A \mapsto K_0(A)$ has a $\Z$-graded extension $A \mapsto K_*(A)$ that is defined on the category of all $C^*$-algebras. It is homotopy-invariant, satisfies Bott periodicity and turns a short exact sequence of $C^*$-algebras into a six-term exact sequence of abelian groups. For a locally compact Hausdorff space $X$ we have a natural isomorphism $K_*(C_0(X)) \cong K^*(X)$ by the Serre-Swan theorem. In the noncommutative geometry \cite{book:Connes} of Alain Connes, $K$-theory of noncommutative $C^*$-algebras led to index theorems for foliations. The classification of amenable $C^*$-algebras \cite{paper:Elliott} is $K$-theoretic. Operator-algebraic $K$-theory also made a recent appearance in condensed matter physics and topological insulators \cite{book:PS-B}.

Twisted $K$-theory admits two complementary definitions that illuminate different aspects of its nature: a first one through non-commutative operator algebras, and a second via stable homotopy theory. By the Serre-Swan theorem and stability of $K$-theory we have 
\[
	K^n(X) \cong K_n(C_0(X) \otimes \bK)
\]
where $\bK$ denotes the compact operators on an infinite-dimensional separable Hilbert space. The algebra $C_0(X) \otimes \bK$ can be interpreted as section algebra of the trivial $\bK$-bundle over $X$. Replacing this by a non-trivial bundle creates a theory that locally agrees with $K$-theory but is ``twisted'' globally, while preserving the structure of being a module over its untwisted version. A result by Dixmier and Douady \cite{paper:DD} states that locally trivial bundles with fibres isomorphic to $\bK$ are classified by third cohomology, i.e.
\[
	[X,B\Aut{\bK}] \cong H^3(X,\Z)\ ,
\]
where $B\Aut{\bK}$ denotes the classifying space of the automorphism group (equipped with the point-norm topology). Given such a bundle $\mathcal{K} \to X$ over a locally compact Hausdorff space $X$, the twisted $K$-groups are the operator-algebraic $K$-theory $K_*(C_0(X,\mathcal{K}))$ of the section algebra $C_0(X,\mathcal{K})$. More generally, one can consider continuous trace $C^*$-algebras instead of locally trivial bundles. Background fields in quantum field theory and string theory are described by Dixmier-Douady invariants~\cite{paper:BV}.

From the viewpoint of stable homotopy theory there is a larger class of twists for $K$-theory than the geometric ones classified by $H^3(X,\Z)$: the tensor product turns $X \mapsto K^*(X)$ into a ring-valued functor. From the perspective of stable homotopy theory $K$-theory is represented by a spectrum, usually denoted $KU$. The ring structure on $K^*(X)$ lifts to $KU$ and turns it into an $E_\infty$-ring spectrum -- the counterpart of a commutative ring in stable homotopy theory. Just as such a ring~$R$ has a group of units formed by the invertible elements in $R$, every $E_\infty$-ring spectrum $E$ has a spectrum of units $gl_1(E)$ formed by spaces $GL_1(E)$, $BGL_1(E)$, etc.\ with an associated cohomology theory $X \mapsto gl_1(E)^*(X)$, for which $gl_1(E)^0(X) \cong GL_1(E^0(X))$. 

The idea of twisted $K$-theory is to replace $KU$ by a bundle of rank $1$-module spectra over $KU$ and then take homotopy classes of sections. An approach to twisted $K$-theory based on stable $\infty$-categories, where this is made precise can be found in \cite[Sec.~3]{paper:AndoBlumbergGepner}. For a commutative ring $R$ the rank $1$-module bundles are classified by $[X, BGL_1(R)]$. Similarly, bundles of $KU$-modules with rank $1$ correspond (up to isomorphism) to elements in 
\[
	gl_1(KU)^1(X) = [X, BGL_1(KU)]\ .
\]
In this picture the geometric twists arise as the ones that factor through a map $K(\Z,3) \to BGL_1(KU)$, which induces an isomorphism on $\pi_3$. Indeed, we have $[X, K(\Z,3)] \cong H^3(X,\Z)$. 

In joint work with Dadarlat the second author has shown that analogous operator-algebraic pictures exist for the higher (non-equivariant) twists of $K$-theory and its localizations \cite{paper:DadarlatPennig1,paper:DadarlatPennig2}. In particular,
\begin{align}
	[X,B\Aut{\Cuntz{\infty} \otimes \bK}] &\cong gl_1(KU)^1(X) \label{eqn:Cuntz}\ ,\\
	[X,B\Aut{M_n^{\otimes \infty} \otimes \bK}] &\cong gl_1(KU[1/n])_+^1(X) \label{eqn:UHF}\ , 	
\end{align} 
where $gl_1(KU[1/n])_+^*$ is the cohomology theory associated to the infinite loop space given by the pullback diagram 
\[
	\begin{tikzcd}[row sep=0.5cm,column sep=0.5cm,every matrix/.append style = {nodes = {font=\footnotesize}}]
		GL_1(KU[1/n])_+ \ar[r] \ar[d] & GL_1(KU[1/n]) \ar[d] \\
		\pi_0(GL_1(KU[1/n])) \cap \Q_+ \ar[r] & \pi_0(GL_1(KU[1/n])) \cong GL_1(\Z[1/n])	
	\end{tikzcd}
\]

Equivariant operator-algebraic $K$-theory is an invariant of $C^*$-dynamical systems, i.e.\ $C^*$-algebras with group actions. The groups $K_*^G(A)$ provide a much finer invariant than $K_*(A)$. Equivariant $K$-theory is relevant for classifying such actions on amenable $C^*$-algebras \cite{paper:GabeSzabo}. In physics it arises through the Verlinde ring in two dimensional chiral conformal field theory which describes the fusion of primary quantum fields \cite{paper:Verlinde}. It appears in the $2$-dimensional conformal Wess-Zumino-Witten models as well as in $3$-dimensional Chern-Simons theory.

The connection between the Verlinde ring of loop groups and equivariant $K$-theory was made by Freed, Hopkins and Teleman \cite{paper:FreedHopkinsTelemanI,paper:FreedHopkinsTelemanII,paper:FreedHopkinsTelemanIII}: Let $G$ be a compact, simple and simply-connected Lie group and let $LG$ be its free loop group, i.e.\ the group of all smooth maps $\gamma \colon S^1 \to G$. Even though $LG$ is infinite-dimensional it has a rich representation theory formed by the positive energy representations at a fixed level $k \in \Z$. After group completion with respect to the direct sum they form a commutative ring $R_k(LG)$ under the fusion product. Many of the most interesting features of $1+1$-dimensional quantum field theories arise from close links to this representation theory as outlined for example in \cite{book:PressleySegal,paper:CareyLangmann,paper:Wassermann,paper:EvansKawahigashi}. Freed, Hopkins and Teleman constructed a ring isomorphism 
\begin{equation} \label{eqn:FHT}
	R_k(LG) \cong \tensor*[^{\tau(k)}]{K}{_{G}^{\dim(G)}}(G)\ ,
\end{equation}
where the right hand side denotes the $G$-equivariant twisted $K$-theory of~$G$ in degree $\dim(G)$ with twist $\tau(k)$ depending on the level $k$ and with respect to the adjoint action of $G$ on itself, which can be realised as the operator-algebraic $K$-theory of an equivariant bundle of compact operators over $G$. The twist is compatible with the multiplication $\mu \colon G \times G \to G$ in the sense that $\mu^*\tau(k) \cong p_1^*\tau(k) + p_2^*\tau(k)$, where $p_1,p_2$ are the projections. The multiplicative structure on the right hand side is the Pontrjagin product induced by a wrong-way map in equivariant twisted $K$-theory associated to $\mu$ (see also \cite{paper:Tu-ringstructure,paper:Waldorf}). 

Since the ring structure of $R_k(LG)$ determines the fusion rules of the Verlinde ring of the chiral conformal field theory associated to the loop group~$LG$, it is a natural question which other invariants of the CFT can be recovered from it. In joint work with Gannon the first named author showed that the full system, and in particular the modular invariant partition function is encoded in the equivariant twisted $K$-theory \cite{paper:EvansGannonI,paper:EvansGannonII}. Other fusion categories like the ones constructed in \cite{paper:TambaraYamagami} have elegant $K$-theoretical descriptions as well as shown in \cite{paper:EvansGannon-loopgroups}.


\subsection{Results in this paper}
Up to isomorphism the geometric twists of the equivariant $K$-theory of a simply-connected group $G$ are classified by the group $H^3_G(G,\Z) \cong H^3(G,\Z) \cong \Z$. As explained in the previous section, and already noted by Atiyah and Segal in \cite{paper:AtiyahSegal}, at least non-equivariantly, stable homotopy theory hands us a much larger group of twists. 

In light of this and the observations from the previous section, several key questions emerge:
\begin{enumerate}[i)]
	\item Is there an equivariant extension of the operator-algebraic models?
	\item Is the equivariant twist $\tau(k)$ that appears in \eqref{eqn:FHT} the shadow of a more general construction involving equivariant higher twists?
	\item What are the consequences for conformal field theories associated to loop groups?
\end{enumerate}
We initiated a programme to investigate the first question in \cite{paper:EvansPennig-circle} starting with circle actions on infinite UHF-algebras. Even though we will leave a complete answer to question~iii) to be discussed in future work, we will briefly come back to it in Sec.~\ref{subsec:HigherDet}. A variation of the second question appeared for example in \cite{paper:Teleman-moduli}, where higher twists of $K_G^*(X)[[t]]$, i.e.\ the power series ring over equivariant $K$-theory, were considered.

In this paper we will focus on a different approach to question~ii) that we developed in \cite{paper:EvansPennig-Twists}. For a simply connected Lie group $G$ the generator of $H^3_G(G,\Z) \cong H^3(G,\Z) \cong \Z$ corresponds to the basic gerbe over $G$ \cite{paper:Meinrenken-basicgerbe, paper:MurrayStevenson}. 

A gerbe over a space $X$ is a higher-categorical generalisation of a line bundle. It is given by a hermitian line bundle $L \to \cG$ over a groupoid $\cG$. This groupoid is Morita equivalent to the trivial groupoid with object space~$X$ and only identity morphisms. To be a gerbe the line bundle $L$ needs to come equipped with a multiplicative structure covering the groupoid multiplication, i.e.\ we have a bundle isomorphism
\[
	\pi_1^*L \otimes \pi_2^*L \to m^*L\ ,
\]   
which is associative in the obvious sense. The maps $\pi_j \colon \cG^{(2)} \to \cG$ and $m \colon \cG^{(2)} \to \cG$ denote the two projections and the multiplication, respectively. For details, we refer the reader to \cite{paper:Murray-bundlegerbes}.

Murray and Stevenson found a construction of the gerbe $L_k \to \cG$ at level $k \in \Z$ for the unitary groups \cite{paper:MurrayStevenson}. In their setting the groupoid $\cG$ can be chosen to be locally compact and given a Haar system. But then $L_k \to \cG$ is an example of a saturated Fell bundle, see \cite[Sec.~2.1]{paper:Kumjian}. As such it has an associated section convolution algebra $C^*(L_k)$ and it turns out that 
\[
	C^*(L_k) \otimes \bK \cong C(G, \mathcal{K}_k)
\] 
for a locally trivial bundle $\mathcal{K}_k \to G$ with fibre $\bK$. Since $L_k \to \cG$ can be equipped with a group action that is compatible with the conjugation action of $G$ on itself, the same turns out to be true for $\mathcal{K}_k \to G$, which allows us to express the $K$-groups from the beginning in terms of operator algebras
\begin{equation} \label{eqn:twistedK-gerbe}
	\tensor*[^{\tau(k)}]{K}{_{G}^{\dim(G)}}(G) \cong K_{\dim(G)}^G(C^*(L_k))\ .
\end{equation}

The fibres of $L_k$ are constructed by applying powers of the determinant functor to eigenspaces of the underlying group elements, where the level enters as the exponent. The crucial observation in \cite{paper:EvansPennig-Twists} is that the only properties of this functor that are actually needed are: 
\begin{enumerate}[(a)]
	\item it maps all objects to~$\C$, 
	\item it naturally transforms direct sums into tensor products and
	\item it is continuous and preserves adjoints.
\end{enumerate}
We obtain interesting new examples of higher twists over $G = SU(n)$ by giving up property (a), but keeping (b) and~(c). More precisely, we change the functor $\left(\extp^{\rm top}\right)^{\otimes (n+k)}$ from the classical setting to an exponential functor 
\[
	F \colon (\Viso,\oplus) \to (\Vgr, \otimes)
\]
on complex inner product spaces and unitary isomorphisms (see Def.~\ref{def:exp_func}). Compared to \cite{paper:EvansPennig-Twists} we will also modify our setting slightly and consider exponential functors that preserve the symmetries on both sides and take values in super-vector spaces. We will see that all examples that have been discussed for example in \cite{paper:Pennig-RMatrices,paper:EvansPennig-Twists} fit much more naturally into this new setup. 

One-dimensional representations provide invertible elements in the representation ring. This is crucial for equivariant bundle gerbes to work. Therefore giving up (a) means that we need to turn higher-dimensional representations into units in equivariant $K$-theory. We achieve this by swapping vector spaces (i.e.\ modules over $\C$) for bimodules over the infinite UHF-algebra
\[
	\MF = \Endo{F(\C^n)}^{\otimes \infty}
\]
and the gerbe $L_k \to \cG$ for a saturated Fell bundle $\cE \to \cG$, whose fibres are invertible $\MF$-$\MF$-bimodules and with multiplication
\[
	\pi_1^*\cE \otimes_{\MF} \pi_2^*\cE \to m^*\cE\ .
\]
One of the upshots of our construction is that equivariance is preserved. In particular, there is a $G$-action on $\cG$ and $\cE$ such that the bundle projection is $G$-equivariant. This turns the section algebra $C^*(\cE)$ into a $\Z/2\Z$-graded $G$-$C^*$-algebra, and we define the equivariant higher twisted $K$-theory of $G = SU(n)$ with twist given by the exponential functor $F$ to be the graded $K$-groups $K_*^G(C^*\cE)$ in analogy to \eqref{eqn:twistedK-gerbe}. It was shown in \cite[Cor.~4.7]{paper:EvansPennig-Twists} that 
\[
	C^*\cE \otimes \bK \cong C(G,\mathcal{A})
\]
for a locally trivial bundle $\mathcal{A} \to G$ with fibres isomorphic to $\MF \otimes \bK$. Up to stabilisation and neglecting equivariance our construction therefore gives a higher twist similar to the ones in \eqref{eqn:UHF}. In contrast to the classical case, where the twist corresponds to an integer (the level), these new twists are parametrised by exponential functors. An in-depth analysis in the ungraded case can be found in~\cite{paper:Pennig-RMatrices}. We return to this point in Sec.~\ref{subsec:HigherDet} where we will see that an exponential functor gives rise to a bundle
\[	
	U(\MF) \to \widetilde{LSU}(n) \to LSU(n)
\]
that is the non-commutative counterpart of the determinant bundle. Classically, this corresponds to the central $U(1)$-extensions classified by the level. 

Replacing the algebra $\C$ by $\MF$ has the effect that the groups $K_*^G(C^*\cE)$ become modules over the localised representation ring 
\[
	K_0^G(\MF) \cong R(G)[F(\rho)^{-1}] =: R_F(G)\ ,
\]
where $G = SU(n)$ and $\rho \colon G \to U(n)$ denotes the inclusion. 

We have shown in \cite[Thm.~5.3]{paper:EvansPennig-Twists} that in the non-graded setting for $G=SU(2)$ and under very mild assumptions on the exponential functor $F$ the equivariant higher twisted $K$-groups satisfy
\begin{align*}
	K_0^G(C^*\cE) &= 0\ , \\
	K_1^G(C^*\cE) &\cong R_F(G)/J_F\ ,
\end{align*}
where $J_F$ is the higher fusion ideal generated by the $SU(2)$ representation corresponding to the character polynomial $\chi_F \in \Z[t,t^{-1}]$ with
\[
	\chi_F = \frac{1}{t - t^{-1}} \,\det \begin{pmatrix}
		F(t) & F(t^{-1}) \\
		1 & 1
	\end{pmatrix}\ .
\]
As explained in \cite[Thm.~5.16]{paper:EvansPennig-Twists} a similar result also holds for $G = SU(3)$ after rationalisation, i.e.\ $K_{\dim(G)}^G(C^*\cE) \otimes \Q \cong (R_F(G) \otimes \Q)/J_{F,\Q}$, where $J_{F,\Q}$ has two generators, whose characters can be expressed in a similar way as~$\chi_F$ above. Both of these results also hold in the graded setting of this paper with the only change that $F(t) \in R(\bT)$ is now a graded representation.  

In the present paper we are now able to complete the picture and compute the rationalised graded equivariant higher twisted $K$-theory for the groups~$G = SU(n)$ and all non-trivial exponential functors $F$. More precisely, we show in our main result, Thm.~\ref{thm:higher_twists}, that the graded higher twisted $K$-groups are
\begin{align} \label{eqn:twistedK-intro}
	K^G_{\dim(G)}(C^*\cE) \otimes \Q &\cong R_F(G) \otimes \Q / J_{F,\Q} \ ,\\
	K^G_{\dim(G)+1}(C^*\cE) \otimes \Q &= 0 \notag 
\end{align}
for an ideal $J_{F,\Q} \subseteq R_F(G) \otimes \Q$. 

With $F = (\extp^{\rm top})^{\otimes (n + k)}$ for $G = SU(n)$ the algebra $\Endo{F(\C^n)}$ is trivially graded (i.e.\ $\MF \cong \C$). As explained in Sec.~\ref{subsec:classical_twists} the $C(G)$-algebra $C^*\cE$ then has graded compact operators as fibres. However, it represents the same twist of $K$-theory as in \eqref{eqn:FHT} in the graded Brauer group $H^1(G,\Z/2\Z) \times H^3(G,\Z)$ of \cite{paper:Maycock}, because $H^1(G,\Z/2\Z)$ vanishes. Indeed, we recover the classical Verlinde ring of $SU(n)$ at level $k$ in accordance with~\eqref{eqn:FHT}.

We prove that the higher fusion ideal $J_{F,\Q}$ has $n-1$ generators constructed as follows: Let $\bT \subseteq G$ be the maximal torus given the by diagonal matrices. Note that its representation ring satisfies $R(\bT) \cong \Z[t_1, \dots, t_n]/(1 - t_1\cdots t_n)$ and $R(G) \cong R(\bT)^W$, where $W \cong S_n$ is the Weyl group of $G$ acting on $R(\bT)$ by permuting the variables. A set of generators of $J_{F,\Q}$ is then given by the character polynomials
\begin{gather*}
	\frac{q_i}{\Delta} \in R(\bT)^W \text{ for } i \in \{0,\dots, n-2\} \text{ with }\\
	q_i(t_1, \dots, t_n) = \det
	\begin{pmatrix}
		F(t_1)\,t_1^i & F(t_2)\,t_2^i & \dots & F(t_n)\,t_n^i \\
		t_1^{n-2} & t_2^{n-2} & \dots & t_n^{n-2} \\
		\vdots & \vdots & \ddots & \vdots \\
		t_1 & t_2 & \dots & t_n 
	\end{pmatrix}\ ,
\end{gather*}
where $\Delta$ is the Vandermonde determinant. So, at least rationally, the equivariant higher twisted $K$-groups are in fact quotient rings of the localised representation ring $R_F(G) \otimes \Q$ by an ideal $J_{F,\Q}$. Generators of the classical fusion ideal for $SU(n)$ at level $k$ have been computed for example by Douglas in \cite[Thm.~1.1]{paper:Douglas-fusionrings} (using $K$-homology instead of $K$-theory). For $F = (\extp^{\rm top})^{\otimes (n+k)}$ our generators correspond to the highest weight representations with weight $(k+i)\,\omega_1$ for $i = \{1, \dots, n-1\}$ by the Weyl character formula where $\omega_1$ is the first fundamental weight of $SU(n)$ (see Rem.~\ref{rem:gen_classical}). It is shown in \cite[Sec.~3.2]{paper:BouwknegtRidout} how one can transform Douglas' set of generators into ours using the Jacobi-Trudy identity. 

As in \cite{paper:AdemCantareroGomez, paper:EvansPennig-Twists} the isomorphisms \eqref{eqn:twistedK-intro} are obtained using the Mayer-Vie\-to\-ris spectral sequence. We compare its $E^1$-page to a cochain complex that computes the $\Waff$-equivariant Bredon cohomology $H^*_{\Waff}(\liet, \cR)$ of $\liet$ with a certain local coefficient system $\cR$ (see Lem.~\ref{lem:cochain_complex}). This comparison does not require rationalisation. It is only the computation of $H^*_{\Waff}(\liet, \cR)$ that is simplified after killing torsion. In addition, the generators of the ideal $J_{F,\Q}$ live in $R_F(G)$. We therefore conjecture that the rationalisation is only a technical difficulty and that the results should also hold integrally.

Gepner discovered in \cite{paper:Gepner} that the fusion ring $R_k(LSU(n))$ is closely related to the cohomology ring of the Grassmann manifolds $G_{k}(\C^{n+k})$ of $k$-dimensional subspaces in $\C^{n+k}$. We have 
\begin{align*}
	G_{k}(\C^{n+k}) &\cong U(n+k)/U(n) \times U(k) \ ,\\
	H^*(G_{k}(\C^{n+k}), \Z) &\cong \Z[\overline{c}_1, \dots, \overline{c}_n]/ (c_{k+1}, \dots, c_{k+n})\ .
\end{align*}
To understand the generators and relations note that there are two non-trivial canonical vector bundles over $G_k(\C^{n+k})$: the tautological bundle and the quotient bundle of its embedding into the trivial bundle. The generators~$\overline{c}_i$ are the Chern classes of the quotient bundle. The ideal can be obtained by expressing the elements $c_j$ in terms of the $\overline{c}_i$ using the identity
\[
	(1 + \overline{c}_1 + \dots + \overline{c}_n) \cdot (1 + c_1 + \dots + c_{k+n}) = 1\ .
\]
In fact, the ideal defining this cohomology ring is an example of a Jacobian ideal, i.e.\ its generators can be obtained as derivatives of a potential $V_{n+k+1}(\overline{c}_1, \dots, \overline{c}_n)$ in such a way that 
\[
	\frac{\partial V_{n+k+1}}{\partial \bar{c}_{n-i}} = (-1)^{n-i+1} c_{k + i + 1} \ .
\] 
for $i \in \{0, \dots, n-1\}$, see \cite{paper:Gepner,paper:Intriligator}. The fusion ring $R_k(LSU(n))$ has a very similar algebraic structure. It is a quotient of $\Z[\overline{c}_1, \dots, \overline{c}_n]$ by the ideal generated from derivatives of the potential $V_{n+k}$ under the additional constraint that $\overline{c}_n = 1$. This constraint may be built into the potential as a perturbation as outlined in \cite{paper:Intriligator}. 

In Sec.~\ref{sec:potential} of this paper we show that a large part of this rigid algebraic structure is preserved when changing from classical equivariant twisted $K$-theory (giving $R_k(LSU(n))$ by \eqref{eqn:FHT}) to equivariant higher twists. We construct a potential $V$ from the character polynomial $F(t_1, \dots, t_n) \in R(\bT) \otimes \Q$ of the exponential functor $F$ and show in Prop.~\ref{prop:potential} that its derivatives generate the higher fusion ideal $J_{F,\Q}$. In fact, our potential is a linear combination of the classical potentials for various levels with coefficients derived from structure constants of the exponential functor. 

We see this as an indication that the close relationship to some Grassmannian will persist in our case. This does not seem too far fetched. Operator-algebraic versions of Grassmannians do exist and have been studied in the past (see for example \cite{paper:Salinas}). Witten \cite{paper:Witten} has given a physical explanation of the relationship between the fusion ring and the cohomology of the Grassmannian manifolds. This then raises the question of whether what may persist could be understood physically.

Finally, we would like to point out another interesting feature of our construction: Just as in \cite{paper:FreedHopkinsTelemanIII} a key step in the computation of the graded equivariant higher twisted $K$-theory is the restriction to the maximal torus $\bT \subset G$. In the classical case the pullback of the basic gerbe with respect to the Weyl map 
\[
	w \colon SU(n)/\bT \times \bT \to SU(n) \quad, \quad ([g], z) \mapsto gzg^*
\]
has been considered in \cite{paper:BeckerMurrayStevenson}. Up to stable isomorphism it agrees with a tensor product of cup-product gerbes \cite[Prop.~5.3]{paper:BeckerMurrayStevenson}. The restriction to $\bT$ is sufficient for our computations, so we will not consider the full Weyl map. Nevertheless, we find a similar tensor product decomposition in Lem.~\ref{lem:equiv_mor_eq}, which takes the following form: the UHF-algebra $\MF$ satisfies the following decomposition into a $\Z/2\Z$-graded tensor product   
\begin{equation} \label{eqn:MF-decomp}
	\MF \cong \MFfac{1} \otimes \dots \otimes \MFfac{n}
\end{equation}
with $\MFfac{i} = \Endo{F(\text{span}\{e_i\})}^{\otimes \infty}$ for the standard basis $\{e_1, \dots, e_n\}$ of $\C^n$. Let $\liet \subset \R^n$ be the Lie algebra of the maximal torus $\bT$. Identifying $\MF$ with its decomposition, the pullback Fell bundle $\cE_\bT \to \bT$ is equivariantly Morita equivalent to 
\begin{equation} \label{eqn:cL}
	\cL = \cL_1 \otimes_{\C} \dots \otimes_{\C} \cL_n\ .
\end{equation}
Each $\cL_i$ is a Fell bundle over $\liet^{[2]}$ with the fibre product taken over the exponential map. The fibre of $\cL_i \to \liet^{[2]}$ over $(x_1, x_2) \in \liet^{[2]}$ is given by the $\MFfac{i}$-$\MFfac{i}$-bimodule
\[
	(F(\text{span}\{e_i\}) \otimes \MFfac{i})^{\otimes q_i(x_2 - x_1)}\ ,
\]
where $q_i \colon \liet \to \R$ is the projection from $\liet \subset \R^n$ onto the $i$th coordinate. This is the operator-algebraic counterpart of the cup-product gerbe decomposition in \cite[Prop.~5.3]{paper:BeckerMurrayStevenson}. This decomposition only holds because we switched to functors that are symmetric monoidal. If $F$ does not preserve symmetries, then the order of the factors in the decomposition \eqref{eqn:MF-decomp} matters and will produce a priori different identifications with $\MF$ (see Rem.~\ref{rem:symmetric_F}). 

\vspace{4mm}

The article is structured as follows: In Section~2 we recall the definition of the equivariant higher twists over $SU(n)$ induced by an exponential functor~$F$ from \cite{paper:EvansPennig-Twists}. We highlight the necessary modifications to make this construction work for symmetric monoidal functors that take values in super-vector spaces. We will also explain that all of the known examples fit much more naturally into this setup.   

The main goal of Section~3 is the computation of the spectral sequence~\eqref{eqn:spec_seq} converging to the equivariant higher twisted $K$-groups $K^G_*(C^*\cE)$. It begins with some background on Lie algebras, root systems and the Weyl alcove. The terms in \eqref{eqn:spec_seq} only involve evaluations of the Fell bundle at points in the maximal torus. Hence, we focus on the restriction $\cE_\bT$ of $\cE$ to $\bT \subset G$ in Sec.~\ref{subsec:maxtorus}. We construct a $\bT$-equivariant equivalence between $\cE_\bT$ and $\cL$ (with $\cL$ as in \eqref{eqn:cL}) and prove in Lem.~\ref{lem:equiv_mor_eq} that it induces a Morita equivalence between $C^*\cE_\bT$ and $C^*\cL$. We also study actions of the normaliser $N(\bT)$ of the maximal torus on these Fell bundles in Sec.~\ref{subsec:WeylGroup}, which give rise to actions of the Weyl group $W$ on the corresponding $K$-theory groups as outlined in Lem.~\ref{lem:WTequiv_mor_eq}. Finally we compare the $E_1$-page of the spectral sequence with the cochain complex giving $H^*_{\Waff}(\liet, \cR)$ in Lem.~\ref{lem:cochain_complex}. After rationalisation we have $H^*_{\Waff}(\liet; \cRQ) \cong H^*_{\Lambda}(\liet; \cRQ)^W$ and the right hand side can be computed using the theory of regular sequences and Koszul complexes, which is done in Lem.~\ref{lem:coh_computation}. The generators of the ideal $J_{F,\Q}$ are constructed in Lem.~\ref{lem:generators}. Finally, Thm.~\ref{thm:higher_twists} summarises the main result.

The first part of Section~4 is devoted to the computation of the potential giving the higher fusion ideal. The main result here is Prop.~\ref{prop:potential}, which recovers the classical potential (up to sign) initially found by Gepner for $G = SU(n)$ at level $k$ for $F = \left(\extp^{\rm top}\right)^{\otimes (n + k)}$. In the second part we construct the $U(\MF)$-bundle \eqref{eqn:UMF-Ures0} over $LG$, sketch an approach to understanding the multiplicative structure on equivariant twisted $K$-theory and highlight some connections to CFT.

\section*{Acknowledgments}
We would like to thank the referee for pointing out an error relating to the structure of the normaliser $N_G(\bT)$ in an early version of the paper. The first named author is supported by an Emeritus Fellowship from the Leverhulme Trust.

\section{Equivariant higher twists over $SU(n)$} \label{sec:equiv_higher_twists}
In this section we extend the definition of equivariant higher twists over $G = SU(n)$ from \cite{paper:EvansPennig-Twists} to take values in complex super-Hilbert spaces. Let $(\Viso, \oplus)$ be the symmetric monoidal category of (ungraded) finite-dimen\-sional complex inner product spaces and unitary isomorphisms with the monoidal structure given by the direct sum. This is a topological groupoid, and we will consider the morphism spaces equipped with their natural topology. Likewise, let $(\Vgr, \otimes)$ be the symmetric monoidal category of finite-dimensional complex inner product super-vector spaces and unitary isomorphisms that preserve the grading. The monoidal structure is given by the graded tensor product and the symmetry is defined on homogeneous elements as follows:
\begin{align*}
	\sigma_{V,W} \colon V \otimes W \to W \otimes V \quad, \quad v \otimes w \mapsto (-1)^{\lvert v \rvert \cdot \lvert w \rvert} w \otimes v \ ,	
\end{align*}
where $\lvert v \rvert$ denotes the degree, i.e.\ with $V = V_0 \oplus V_1$ we have $\lvert v \rvert = i$ for $v \in V_i$. Unless otherwise stated, we will always consider tensor products to be graded. We will also consider $\Vgr$ as a topologically enriched category with the natural topology on the morphism spaces. 

\begin{definition} \label{def:exp_func}
An \emph{exponential functor} $F \colon \Viso \to \Vgr$ is a continuous symmetric monoidal functor from $(\Viso, \oplus)$ to $(\Vgr, \otimes)$, which preserves duals (i.e.\ there is a natural isomorphism $F(V)^* \cong F(V^*)$) and such that the $S^1$-representation $F(\C)$ has only positive characters. In particular, $F$ comes equipped with two natural isomorphisms 
\begin{gather*}
	\tau_{V,W} \colon F(V \oplus W) \to F(V) \otimes F(W) \quad \text{and} \quad \iota \colon F(0) \to \C \ , 
\end{gather*}
which make the obvious unitality and associativity diagrams commute. Being symmetric implies that  
\[
	\tau_{W,V} \circ F(\sigma^{\oplus}_{V,W}) = \sigma^{\otimes}_{F(V), F(W)} \circ \tau_{V,W}\ ,
\]
where $\sigma^{\oplus}, \sigma^{\otimes}$ denote the respective symmetry transformations for the direct sum and the tensor product (on super-vector spaces).
\end{definition}


\begin{example} \label{ex:det}
	The classical twists arise in this setting from (graded) powers of the determinant functor. Let $l \in \N$ and define $F$ for $V \in \Viso$ to be 
	\[
		\det\,\!\!^{\otimes l}(V) = \left(\extp^{\rm top} V\right)^{\otimes l}\ ,
	\] 
	where we equip the exterior algebra with its natural $\Z/2\Z$-grading. This means that $\det\,\!\!^{\otimes l}(V)$ is purely odd if $\dim(V)\cdot l$ is odd and purely even otherwise. Another example arises from the full exterior algebra. More generally, for a trivially graded real inner product space $X$, we define
	\[
		F^X(V) = \bigoplus_{m \in \N_0} X_{\C}^{\otimes m} \otimes \extp^m V 
	\]
	again equipped with its natural $\Z/2\Z$-grading. The inner product on the $m$th exterior power is 
	\[
		\langle \xi_1 \wedge \dots \wedge \xi_m, \eta_1 \wedge \dots \wedge \eta_m \rangle = \det\left( \langle \xi_i, \eta_j \rangle_{i,j} \right)
	\]
	and the summands are orthogonal. With this definition we obtain a natural isomorphism $F^X(V^*) \cong F^X(V)^*$. We refer the reader to \cite[Sec.~2.2]{paper:Pennig-RMatrices} for the definitions of $\tau_{V,W}$ and $\iota$ and further details.
	
	Note that if $F$ does not satisfy the character condition in Def.~\ref{def:exp_func}, $\det^{\otimes l} \otimes F$ does for a suitable choice of $l \in \N$. 
\end{example}

\subsection{Graded $C^*$-algebras and graded Morita equivalences}
To an exponential functor $F \colon \Viso \to \Vgr$ we will associate a $\Z/2\Z$-graded $C^*$-algebra~$\MF$. A $\Z/2\Z$-grading on a $C^*$-algebra $A$ is an order $2$ automorphism, i.e.\ $\gamma \in \Aut{A}$ such that $\gamma^2 = \id{A}$. It induces a direct sum decomposition $A = A^{(0)} \oplus A^{(1)}$ as a Banach space with 
\[
	A^{(i)} = \{ a \in A \ |\ \gamma(a) = (-1)^{i}\,a\}\ .
\]  
If there is a self-adjoint unitary $s \in M(A)$ such that $\gamma = \Ad_s$, then we call the grading inner. A graded $*$-homomorphism $\varphi \colon A \to B$ is one that intertwines the grading automorphisms $\gamma_A$ and $\gamma_B$. 

Given a graded $C^*$-algebra $A$ a graded Hilbert $A$-module $E$ is a right Hilbert $A$-module equipped with a linear bijection $S_E \colon E \to E$ such that $S_E^2 = \id{E}$ and for all $v,w \in E$, $a \in A$ 
\[
	S_E(v\,a) = S_E(v)\,\gamma_A(a) \quad \text{and} \quad \langle S_E(v), S_E(w) \rangle_A = \gamma_A(\langle v, w \rangle_A) \ .
\]
Defining $E^{(i)} = \{ v \in E \ | \ S_E(v) = (-1)^i\,v \}$ we have $E = E^{(0)} \oplus E^{(1)}$. If $E$ is a graded Hilbert $A$-module, then $\Ad_{S_E}$ defines an order $2$ automorphism of the compact operators $\rcpt{A}{E}$ turning them into a graded $C^*$-algebra with an inner grading. A graded Morita equivalence between graded $C^*$-algebras $A$ and $B$ is a graded right Hilbert $B$-module $E$, which is full in the sense that $\langle E^{(i)}, E^{(j)} \rangle \subseteq B^{(i+j)}$ is dense, together with an isomorphism $\varphi \colon A \to \rcpt{B}{E}$ of graded $C^*$-algebras. Given graded $C^*$-algebras $A$, $B$ and $C$ together with an $A$-$B$-Morita equivalence $E$ and a $B$-$C$-equivalence $F$, there is an internal tensor product $E \otimes_B F$, which is an $A$-$C$-Morita equivalence. This is formed as described in \cite[Sec.~1.2.3]{book:JensenThomsen} with the grading operator given by $S_E \otimes S_F$. 

Given two graded $C^*$-algebras $A$ and $B$ we may equip the algebraic tensor product $A \odot B$ with the graded multiplication given on homogeneous elements $a^{(i)} \in A^{(i)}, b^{(j)} \in B^{(j)}, c^{(k)} \in A^{(k)}, d^{(l)} \in B^{(l)}$ by
\[	
	(a^{(i)} \otimes b^{(j)}) \cdot (c^{(k)} \otimes d^{(l)}) = (-1)^{jk}\,(a^{(i)} \cdot c^{(k)}) \otimes (b^{(j)} \cdot d^{(l)})
\]
and the graded (minimal) tensor product $A \otimes B$ is the (minimal) completion of it \cite[Sec.~14.4]{book:Blackadar}. For a homogeneous element $a$ we will sometimes denote the degree of $a$ by $\lvert a \rvert$.

Let $V$ be a finite-dimensional complex super-vector space with inner product. The endomorphism algebra $\Endo{V}$ is a graded $C^*$-algebra with inner grading. If $W$ is another such space there is a natural isomorphism of graded $C^*$-algebras (described in \cite[Prop.~14.5.1]{book:Blackadar})
\begin{equation} \label{eqn:graded_end}
	\Endo{V} \otimes \Endo{W} \to \Endo{V \otimes W}\ ,
\end{equation}
which is compatible with the monoidal symmetry on both sides. Let $W$ be a super-vector space equipped with a unitary $G$-representation that preserves the grading. Then $G$ acts by degree-preserving automorphisms on $\Endo{W}$. Consider the colimit $\Endo{W}^{\otimes \infty}$ with connecting maps
\begin{equation} \label{eqn:End_colimit}
	\Endo{W}^{\otimes n} \to \Endo{W}^{\otimes (n+1)} \quad , \quad T \mapsto T \otimes 1\ .
\end{equation}
Denote the grading operator on $W$ by $s$ and the ungraded tensor product by $\overline{\otimes}$ (we will not use this non-standard notation in the rest of the paper). By \cite[Prop.~14.5.1]{book:Blackadar} the map
\[
	\Endo{W} \otimes \Endo{W} \to \Endo{W} \overline{\otimes} \Endo{W} \ , \  T_1 \otimes T_2 \mapsto T_1s^{\lvert T_2 \rvert}\,\overline{\otimes}\, T_2
\]
provides a $*$-isomorphism between the graded tensor product and the algebra $\Endo{W} \overline{\otimes} \Endo{W}$ equipped with inner grading given by $\Ad_s \otimes \Ad_s$. Since the identity is even, the sequence \eqref{eqn:End_colimit} is isomorphic to 
\[
	\Endo{W}^{\overline{\otimes} n} \to \Endo{W}^{\overline{\otimes} (n+1)} \quad , \quad T \mapsto T \otimes 1
\]
with the grading on $\Endo{W}^{\overline{\otimes} n}$ given by $\Ad_{s^{\overline{\otimes} n}}$. Hence, we can identify the colimit $\Endo{W}^{\otimes \infty}$ with an ungraded tensor product as well, on which the grading automorphism is approximately inner and acts by conjugation with~$s$ on each tensor factor. The $G$-action obtained as an infinite tensor product of the action on $\Endo{W}$ commutes with this automorphism.

The construction outlined in the following section gives rise to a $\Z/2\Z$-graded $G$-$C^*$-algebra $C^*(\cE)$ for $G = SU(n)$. We will compute its graded $K$-theory in the sense of Kasparov \cite[Def.~2.3]{paper:Kasparov}, i.e.\ we define for a graded $G$-$C^*$-algebra $A$
\[
	K_i^G(A) = KK_G^i(\C, A)\ .
\]
This functor is homotopy-invariant, stable, continuous and has six-term exact sequences for semi-split exact sequences of graded $G$-$C^*$-algebras. Let $W$ be a finite-dimensional super-Hilbert space equipped with a representation by $G$ that preserves the grading. Note that $W^{\otimes n}$ provides a $G$-equivariant graded $\Endo{W}^{\otimes n}$-$\C$-Morita equivalence that gives rise to 
\[
	K_0^G(\Endo{W}^{\otimes n}) \cong K_0^G(\C) \cong R(G)\ .
\]
This isomorphism intertwines the connecting maps in \eqref{eqn:End_colimit} with the multiplication by the virtual representation $W_{\rm gr} := W^{(0)} - W^{(1)} \in R(G)$. In the colimit we therefore obtain the localisation 
\[
	K_0^G(\Endo{W}^{\otimes \infty}) \cong R(G)[W_{\rm gr}^{-1}]\ .
\]

\subsection{Construction of the Fell bundle}
In this section we will recall the construction of the higher twist from \cite{paper:EvansPennig-Twists}. The twist is given as a saturated Fell bundle $\cE$ over a (locally compact Hausdorff) groupoid $Y^{[2]}$. We refer the reader to \cite[Def.~2.6]{paper:BussExel} for the definition of Fell bundles. The fibres of $\cE \to Y^{[2]}$ over the units of the groupoid will be $C^*$-algebras and over general elements Morita equivalences between domain and range algebra. We therefore start by defining those. 

Let $\rho \colon G \to U(n)$ for $G = SU(n)$ be the standard representation (i.e.\ $\rho$ is the homomorphism given by inclusion). It gives rise to the (grading-preserving) unitary $G$-representation 
\[
	F(\rho) \colon G \to U(F(\C^n))
\] 
on the super-vector space $F(\C^n)$. As in \eqref{eqn:End_colimit} the UHF-algebra 
\[
	\MF = \Endo{F(\C^n)}^{\otimes \infty}
\]
is a $\Z/2\Z$-graded $G$-$C^*$-algebra with $G$-action given by $\left(\Ad_{F(\rho)}\right)^{\otimes \infty}$. For any subspace $V \subset \C^n$ the right Hilbert $\MF$-module 
\begin{equation} \label{eqn:subspace_module}
		\cV = F(V) \otimes \MF
\end{equation}
is in fact a graded $\MF$-$\MF$ Morita equivalence bimodule. The left multiplication by $\MF$ is induced by the canonical isomorphism of the compact operators $\rcpt{\MF}{\cV}$ with $\Endo{F(V)} \otimes \MF$ and the composition
\begin{equation} \label{eqn:left_mul}
\begin{tikzcd}[column sep=0.6cm, every matrix/.append style = {font = \footnotesize}]
	\Endo{F(V)} \otimes \MF \ar[r,"\cong"] & \Endo{F(V)} \otimes \Endo{F(V) \otimes F(V^\perp)}^{\otimes \infty} \ar[r,"\cong"] & \MF\ ,
\end{tikzcd}
\end{equation}
where the first map applies $(\Ad_{\tau_{(V,V^\perp)}})^{\otimes \infty}$ to $\MF$ and the second map is given by shifting the tensor factors accordingly. Note that permuting the tensor factors using the symmetry of $\Vgr$ involves signs. Other than this the construction is the same as in \cite[Lem.~3.7]{paper:EvansPennig-Twists}.

The vector space $F(V)$ carries a left action by $\Endo{F(V)}$. This turns into a right action by $\Endo{F(V)}$ on the dual space $F(V)^* \cong F(V^*)$. Therefore the opposite bimodule of~$\cV$ is given by 
\[
	\cV^{\text{op}} = F(V)^* \otimes \MF \cong F(V^*) \otimes \MF\ ,
\]
where the left action of $\MF$ only acts on $\MF$ by left multiplication, but the right action makes use of the isomorphism $\Endo{F(V)} \otimes \MF \cong \MF$. Note that the notation ``$\op$'' will be reserved for the opposite bimodule, not for the opposite grading. 

The Morita equivalence bimodules described above form the fibres of a Fell bundle over a groupoid that we now construct: For an element $g \in G$ denote by $\EV{g}$ the set of eigenvalues of $g$. The eigenspace corresponding to the eigenvalue $\lambda \in \EV{g}$ will be denoted by~$\Eig{g}{\lambda}$. Let 
\begin{equation} \label{eqn:Y}
	Y = \{ (g,z) \in G \times S^1 \setminus \{1\}\ |\ z \notin \EV{g}\}
\end{equation}
and let $Y^{[2]}$ be the fibre product of $Y$ with itself over $G$, i.e.\ a point $(g,z_1,z_2) \in G \times (S^1 \setminus \{1\})^2$ is in $Y^{[2]}$ if and only if $z_i \notin \EV{g}$ for $i \in \{1,2\}$. This is a groupoid with respect to the composition 
\[
	(g,z_1,z_2) \cdot (g,z_2,z_3) = (g,z_1,z_3)\ .
\]
Choose a total order on $S^1 \setminus \{1\}$ by declaring $z_1 < z_2$ if the arc from $z_1$ to~$z_2$ in $S^1 \setminus \{1\}$ runs counterclockwise. The groupoid has a decomposition into three disjoint components
\[
	Y^{[2]} = Y^{[2]}_+\ \amalg\ Y^{[2]}_0\ \amalg\ Y^{[2]}_-\ ,
\]
where $Y^{[2]}_+$ contains all points $(g,z_1,z_2)$ with $z_1 < z_2$ such that there is $\lambda \in \EV{g}$ with $z_1 < \lambda < z_2$. The space $Y^{[2]}_-$ is defined similarly, but with $z_1 > z_2$ and $Y^{[2]}_0$ is the space with no eigenvalues between $z_1$ and $z_2$. In the following we will denote the restrictions of the Fell bundle $\cE$ to these three subspaces by $\cE_+$, $\cE_0$ and $\cE_-$ respectively.

The bundle $\cE_+$ constructed in \cite{paper:EvansPennig-Twists} is defined to be the locally trivial $\Z/2\Z$-graded right Hilbert $\MF$-module bundle with fibre over $(g,z_1,z_2) \in Y^{[2]}_+$ given by
\[
	\cE_{(g,z_1,z_2)} =  F\!\left(E(g,z_1,z_2)\right) \otimes \MF \quad
	\text{ with } E(g,z_1,z_2) = \bigoplus_{z_1 < \lambda < z_2 \atop \lambda \in \EV{g}}\Eig{g}{\lambda}\ . 
\]
By \cite[Cor.~3.8]{paper:EvansPennig-Twists} the endomorphism bundle of $\cE_+$ has a trivialisation that restricts to the left $\MF$-module structure on $\cE_{(g,z_1,z_2)}$ from \eqref{eqn:left_mul} in each fibre. The monoidal natural transformation of $F$ provides an isomorphism 
\[
	F\!\left(E(g,z_1,z_2)\right) \otimes F\!\left(E(g,z_2,z_3)\right) \cong F\!\left(E(g,z_1,z_3)\right) \ .
\]
Over $Y^{[2]}_+$ we therefore obtain a multiplication 
\[
	\cE_{(g,z_1,z_2)} \otimes_{\MF} \cE_{(g,z_2,z_3)} \to \cE_{(g,z_1,z_3)}\ .
\]
Note that the proof of its associativity in \cite[Lem.~3.7]{paper:EvansPennig-Twists} and continuity in \cite[Cor.~3.8]{paper:EvansPennig-Twists} carry over verbatim to the graded case. Define $\cE_0$ to be the trivial bundle with fibre $\MF$ and let 
\[
	\cE_{(g,z_1,z_2)} = \left( \cE_{(g,z_2,z_1)} \right)^{\text{op}} \quad \text{ for } (g,z_1,z_2) \in Y^{[2]}_-\ . 
\]

The spaces $E(g,z_1,z_2)$ form the fibres of a vector bundle $E \to Y^{[2]}_+$ by \cite[Sec.~3]{paper:MurrayStevenson}. Thus, 
\[
	\cE_+ \cong F(E) \otimes \MF\ .
\] 
Hence, the grading is a continuous operation on $\cE_+$ and similarly on $\cE_-$. The bundle $\cE_0$ is trivial, so the grading is constant. 

The proof of \cite[Thm.~3.3]{paper:EvansPennig-Twists} is based on properties of the two inner products on imprimitivity bimodules which also hold in the graded case, so we obtain a Fell bundle $\cE \to Y^{[2]}$ as in \cite[Cor.~3.12]{paper:EvansPennig-Twists} with a grading $S_\cE \colon \cE \to \cE$ turning each fibre $\cE_{(g,z_1,z_2)}$ into a graded $\MF$-$\MF$ Morita equivalence bimodule in such a way that the Fell bundle multiplication 
\[
	\cE_{(g,z_1,z_2)} \otimes_{\MF} \cE_{(g,z_2,z_3)} \to \cE_{(g,z_1,z_3)}
\]
is compatible with the grading.

Let $h \in G$. In the standard representation $\rho(h) \colon \C^n \to \C^n$ restricts to an isomorphism $\Eig{g}{\lambda} \to \Eig{hgh^{-1}}{\lambda}$. Thus, $G$ acts on $Y^{[2]}$ by $h \cdot (g,z_1,z_2) = (hgh^{-1}, z_1,z_2)$. This action lifts to a continuous grading-preserving action of $G$ on $\cE$ as described in \cite[Cor.~3.6]{paper:EvansPennig-Twists}. Altogether we obtain a $\Z/2\Z$-graded $G$-equivariant saturated Fell bundle $\cE \to Y^{[2]}$. 

The $C^*$-algebra associated to $\cE$ is constructed in the same way as in \cite[Sec.~4]{paper:EvansPennig-Twists}, but we will point out where the grading enters: The algebra $A = C_0(Y,\MF)$ is now $\Z/2\Z$-graded. Because the Fell bundle multiplication is compatible with the grading, the right Hilbert $A$-module $L^2(\cE)$ is graded as well, which induces a $\Z/2\Z$-grading on the adjointable bounded $A$-linear operators $\rbdd{A}{L^2(\cE)}$. There is a well-defined graded $*$-homomorphism 
\[
	C_c(Y^{[2]},\cE) \to \rbdd{A}{L^2(\cE)}\ ,
\]
where the domain acts as convolution operators on $L^2(\cE)$. We define $C^*\cE$ as the norm closure of the domain in $\rbdd{A}{L^2(\cE)}$. This is a Fell bundle $C^*$-algebra with the grading as extra structure. Therefore \cite[Lem.~4.2]{paper:EvansPennig-Twists} is still valid and shows that $C^*\cE$ is a continuous $C(G)$-algebra with the graded $C^*$-algebras $\cE_g$ as its fibres and \cite[Lem.~4.6]{paper:EvansPennig-Twists} provides a graded Morita equivalence between $\cE_g$ and $\MF$. Moreover, \cite[Lem.~4.3]{paper:EvansPennig-Twists} produces graded $G$-equivariant Morita equivalences. 

We define the equivariant higher twisted $K$-theory of $G=SU(n)$ with twist given by the Fell bundle $\cE$ by
\[
	K^i_{G,\cE}(G) := K_i^G(C^*\cE)\ .
\]
The Fell condition \cite[Def.~4.5]{paper:EvansPennig-Twists} implies that $C^*\cE \otimes \bK$ is isomorphic as a $C(G)$-algebra to the section algebra $C(G,\cA)$ of a locally trivial bundle $\cA \to G$. If we transfer the $G$-action and grading from $C^*\cE \otimes \bK$ to $C(G,\cA)$ through this isomorphism, then compatibility with the $C(G)$-algebra structure implies that the grading acts fibrewise on $\cA$ and the $G$-action covers the conjugation action. Thus, to compute $K_{G,\cE}^i(G)$ we may use the Mayer-Vietoris spectral sequence constructed in \cite[Prop.~4.9]{paper:EvansPennig-Twists}.

This spectral sequence computes these $K$-groups from the representation rings $R(H)$ for certain subgroups $H \subseteq G$. More precisely, we need localisations of these rings defined as follows: If $V$ is a finite-dimensional unitary representation of $G$, then $F(V)$ is again a finite-dimensional unitary representation. Since $F$ is exponential it gives rise to a monoid homomorphism 
\[
	F \colon (R(H),\oplus) \to (R(H), \otimes)
\]
for any subgroup $H \subseteq G$, which we continue to denote by $F$ by slight abuse of notation. For a subgroup $H \subseteq G$ we define
\begin{align*}
	R_F(H) &= R(H)[F(\left.\rho\right|_H)^{-1}] \ .
\end{align*}
where $\rho \colon G \to U(n)$ denotes the standard representation.

Finally, we also need to see that $K_*^G(C^*\cE)$ is a module over $K_0^G(\MF)$. It suffices to see that the tensor embedding $C^*\cE \to C^*\cE \otimes \MF$ induces an isomorphism on $K_*^G$. As in \cite[Prop.~4.11]{paper:EvansPennig-Twists} this problem can be reduced to checking that $C(X, \MF) \to C(X, \MF) \otimes \MF$ with $f \mapsto f \otimes 1$ induces an isomorphism for a compact Hausdorff $G$-space $X$. By treating grading and $G$-action together as a $G \times \Z/2\Z$-action this follows in the same way as in \cite[Lem.~4.10]{paper:EvansPennig-Twists}, but it can also be shown directly as follows: It suffices to see that the first factor embedding 
\[
	l \colon \MF \to \MF \otimes \MF \quad , \quad a \mapsto a \otimes 1
\] 
is asymptotically $G$-unitarily equivalent to an isomorphism through grading-preserving unitaries. This can be achieved as in  \cite[Lem.~2.3]{paper:EvansPennig-circle}. There is an equivariant isomorphism $\varphi \colon \MF \otimes \MF \to \MF$ that alternates between the two tensor factors and preserves the grading. Hence, it suffices to show that there is a path $u \colon [0,1) \to U(\MF)$ such that for all $a \in \MF$ 
\[
	\lim_{t \to 1} \lVert u_t\,(\varphi \circ l)(a)\,u_t^* - a \lVert = 0 \ .
\]
Note that the subgroup $U(F(\C^n) \otimes F(\C^n))^{G \times \Z/2\Z}$ of unitaries fixed by the grading and the $G$-action decomposes into a product of unitary groups, which is path-connected. Therefore there is a path 
\[
	v \colon [0,1] \to U(F(\C^n) \otimes F(\C^n))\ ,
\]
which is $G$-invariant, preserves the grading and connects the identity map to the one interchanging the two tensor factors. Now we proceed as in the proof of \cite[Lem.~2.3]{paper:EvansPennig-circle} with the construction of $u_t$. The diagram on \cite[p.~922]{paper:EvansPennig-Twists} shows that the induced multiplication on $K_0^G(\MF)$ corresponds to the ring structure of $R_F(G)$.

\subsubsection{Classical twists} \label{subsec:classical_twists}
	Following \cite[Theorem~1]{paper:FreedHopkinsTelemanIII} classical twists for $SU(n)$ at level $k \in \N_0$ are given by $(h^\vee + k)$-fold tensor powers of the basic gerbe, where $h^\vee = n$ is the dual Coxeter number of $SU(n)$. In our setup this situation corresponds to the exponential functor \[
			F = \left(\extp^{\rm top}\right)^{\otimes (n+k)}\ ,
		\]
		where our construction boils down to the one from \cite{paper:MurrayStevenson}. By \cite[Lem.~4.2]{paper:EvansPennig-Twists} the $C^*$-algebra $C^*(\cE)$ is a continuous $C(G)$-algebra with fibre $C^*(\cE_g)$, where $\cE_g$ is the restriction of $\cE$ to the subgroupoid of $Y^{[2]}$ over $g$, i.e.\ 
		\[
			Y_g^{[2]} \cong \left\{ (z_1, z_2) \in (S^1 \setminus \{1\})^2 \ | \  z_i \notin \EV{g} \text{ for } i \in \{1,2\} \right\}\ .
		\]
		For $z_1 \leq z_2$ let $E(z_1,z_2)$ be the direct sum of the eigenspaces $\Eig{g}{\lambda}$ with $z_1 < \lambda < z_2$ (which is the zero vector space if there are no eigenvalues in-between). The fibre $(\cE_g)_{(z_1, z_2)}$ is 
		\[
			(\cE_g)_{(z_1, z_2)} = \begin{cases}
				\left(\extp^{\rm top} E(z_1,z_2)\right)^{\otimes (n+k)} & \text{if } z_1 \leq z_2 \ ,\\
				\left(\left(\extp^{\rm top} E(z_2,z_1)\right)^{\otimes (n+k)}\right)^* & \text{if } z_1 > z_2\ .
			\end{cases}
		\]
		Let $X_g = S^1 \setminus (\{1\} \cup \EV{g})$, $z_0 \in X_g$ and let $\sigma \colon X_g \to Y_g^{[2]}$ be defined by $\sigma(z) = (z,z_0)$. Then $L^2(\cE)$ is a continuous field of Hilbert spaces over $Y$ with fibre over $(g,z_0)$ isomorphic to $H = L^2(X_g, \sigma^*\cE_g)$. Let
		\begin{gather*}
			\kappa_g \colon Y_g^{[2]} \to \Z/2\Z \ ,\\
			(z_1, z_2) \mapsto
			\begin{cases}
				(n+k)\,\dim(E(z_1,z_2)) \mod 2 & \text{if } z_1 \leq z_2\ , \\
				(n+k)\,\dim(E(z_2,z_1)) \mod 2 & \text{else} \ .
			\end{cases}
		\end{gather*}
		This is a continuous groupoid homomorphism. The Hilbert space $H$ has a $\Z/2\Z$-grading, where $H_+$ is the closure of compactly supported functions with support in $(\kappa_g \circ \sigma)^{-1}(0)$. Moreover, $C^*(\cE_g)$ is the closure of the convolution algebra given by operators with compactly supported sections of $\cE_g$ as their integral kernels, i.e.\ the algebra of compact operators on $H$. A compact operator that corresponds to a compactly supported section of $\cE_g$ is even with respect to the $\Z/2\Z$-grading on $C^*(\cE)$ if the section is supported in $\kappa_g^{-1}(0)$ and odd if the support is in $\kappa_g^{-1}(1)$. Hence, the even operators are the ones preserving the decomposition $H_+ \oplus H_-$, whereas the odd ones map $H_+$ to $H_-$ and vice versa. If $n+k$ is even, $H_- = 0$ and the grading is trivial.
		
		The map $\kappa_g$ is the restriction of a continuous map $\kappa \colon Y^{[2]} \to \Z/2\Z$ to the fibre over $g$, which is defined in the same way as above. Consider the quotient of $Y \times \Z/2\Z$ by the equivalence relation
		\[
			(y_1, m) \sim (y_2, n) \text{ if and only if } \pi(y_1) = \pi(y_2) \text{ and } n = \kappa(y_1,y_2) + m 
		\] 
		This gives a principal $\Z/2\Z$-bundle over $G$ that represents a class $[\kappa] \in H^1(G,\Z/2\Z)$ in accordance with the fact that the classical twists are up to isomorphism classified by $H^1(G,\Z/2\Z) \times H^3(G,\Z)$. Since we are considering $G = SU(n)$, we have $H^1(G,\Z/2\Z) = 0$ and $[\kappa]$ vanishes in cohomology.

\section{The spectral sequence computing higher twisted $K$-theory}
We start this section by recalling a few basic facts about the geometry underlying the conjugacy classes of $G = SU(n)$. Let $\ell = n-1$ be the rank of~$G$. Denote by $\bT \subset SU(n)$ the maximal torus consisting of diagonal matrices, let $\liet$ be its Lie algebra. Note that 
\[
	\liet = \{ (\xi_1, \dots, \xi_n) \in \R^n \ | \ \ \xi_1 + \dots + \xi_n = 0 \} \subset \R^n \ .
\]
Let $\Lambda = \ker(\exp \colon \liet \to \bT) \subset \liet$ be the integral weight lattice, and let $\Lambda^*$ be its dual lattice\footnote{We absorb the factor $2\pi i$ into the definition of $\exp$.}. We have 
\[
	\Lambda = \{ (\lambda_1, \dots, \lambda_n) \in \Z^n\ | \ \lambda_1 + \dots + \lambda_n = 0\} \subset \Z^n \ .
\]
Denote by $\rscal{\,\cdot\,}{\,\cdot\,}{\lieg}$ the basic inner product on the Lie algebra $\lieg$ of $G$. Choose a collection of simple roots $\alpha_1, \dots, \alpha_\ell \in \liet^*$ and define
\[
	\liet_+ = \{ \xi \in \liet \ |\ \rscal{\alpha_j}{\xi}{\lieg} \geq 0 \ \forall j \in \{1, \dots, \ell\} \}\ .
\]
This is the corresponding positive Weyl chamber. Let $\alpha_0 \in \Lambda^*$ be the lowest root. The intersection of $\liet_+$ with the half-plane defined by $\rscal{\alpha_0}{\xi}{\lieg} \geq -1$ is the fundamental alcove of $G$. For $SU(n)$ we can take $\alpha_i(\xi) = \xi_i - \xi_{i+1}$ for $i \in \{1, \dots, \ell\}$ as the simple roots. The lowest root is given by $\alpha_0(\xi) = \xi_n - \xi_1$. The vertices of the Weyl alcove are then given by the origin $\mu_0$ and the points 
\[
	\mu_k = \left(\underbrace{\frac{k}{n}, \dots, \frac{k}{n}}_{n-k \text{ times}},\underbrace{\frac{k-n}{n}, \dots,  \frac{k-n}{n}}_{k \text{ times}} \right)
\]
for $k \in \{1, \dots, \ell\}$. Note that the vertex $\mu_k \in \liet$ lifts the central element $\omega^k 1_n \in Z(SU(n))$, where $\omega = e^{2\pi i / n}$. The simplex obtained as the convex hull of the set $\{\mu_0, \dots, \mu_\ell\} \subset \liet$ parametrises the conjugacy classes in~$G$. We can identify it with the standard $\ell$-simplex 
\[
	\Delta^\ell = \left\{ (t_1, \dots, t_\ell) \in \R^{\ell}\ \mid\ \sum_{i=1}^\ell t_i \leq 1 \text{ and } t_j \geq 0 \ \forall j \in \{1,\dots, \ell\} \right\}
\]
by mapping the point $(t_1, \dots, t_{\ell})$ to $\sum_{i=1}^\ell t_i \mu_i$. In this way $(0,\dots,0)$ corresponds to $\mu_0$. Let $W = N(\bT)/\bT \cong S_n$ be the Weyl group of $G$. 

For a non-empty subset $I \subset \{0,\dots, \ell\}$ let $\Delta_I \subset \Delta^\ell$ be the closed subsimplex spanned by the vertices in $I$. Denote by $\xi_I \in \Delta^\ell \subset \lieg$ the barycentre\footnote{Taking any other point in the interior of the simplex $\Delta_I$ will not change $G_I$ up to isomorphism.} of $\Delta_I$ and let $G_I$ be the centraliser of $\exp(\xi_I) \in G$. Let $W_I \subset W$ be the stabiliser of $\exp(\xi_I)$. Note that $W_{\{0\}} = \dots = W_{\{\ell\}} = W$. Let $I = \{i_0, \dots, i_r\}$ with $0 \leq i_0 < i_1 < \dots < i_r \leq \ell$. If $\lvert I \rvert = 1$, then $G_I = G_{\{i_0\}} \cong SU(n)$. Otherwise,  
\begin{equation} \label{eqn:GI-iso}
	G_I \cong U(i_r - i_{r-1}) \times \dots \times U(i_1 - i_0) \times U(n - (i_r - i_0)) \cap SU(n) \ .
\end{equation}
The groups $G_I$ for $G=SU(3)$ and all $I \subset \{0,1,2\}$ are shown in Fig.~\ref{fig:barycentres}.

\begin{figure}[htp]
\centering
\begin{tikzpicture}[scale=3.8]
	\coordinate (0) at (0,0);
	\coordinate (alpha2) at (0,{sqrt(2)});
	\coordinate (alpha1) at ({sqrt(2)*sin(120)}, {sqrt(2)*cos(120)});
	\coordinate (alpha3) at ($(alpha1)+(alpha2)$);
	
	\coordinate (mu0) at ($1/3*(alpha1) + 2/3*(alpha2)$);
	\coordinate (mu1) at ($2/3*(alpha1) + 1/3*(alpha2)$);

	\coordinate (b0) at ($1/2*(mu0)$);
	\coordinate (b1) at ($1/2*(mu1)$);
	\coordinate (b2) at ($(b0) + (b1)$);

	\coordinate (c) at ($1/3*(mu0) + 1/3*(mu1)$);
	
	\draw[fill=blue!10] (0,0) -- ($(mu0)$) -- ($(mu1)$) -- cycle;
	
	\draw ($(mu0)$) -- ($(mu1)$);
	\draw (0,0) -- ($(mu0)$);
	\draw (0,0) -- ($(mu1)$);
	
	\draw[fill=black, black, thick] (0,0) circle (0.3pt) node[below,black] {$SU(3)$};
	\draw[fill=black, black, thick] ($(mu0)$) circle (0.3pt) node[above,black] {$SU(3)$};
	\draw[fill=black, black, thick] ($(mu1)$) circle (0.3pt) node[below,black] {$SU(3)$};

	\draw[fill=black, black, thick] ($(b0)$) circle (0.3pt) node[left,black] {$U(2)$};
	\draw[fill=black, black, thick] ($(b1)$) circle (0.3pt) node[below,black] {$U(2)$};
	\draw[fill=black, black, thick] ($(b2)$) circle (0.3pt) node[right,black] {$U(2) \cong U(2) \times U(1) \cap SU(3)$};

	\draw[fill=black, black, thick] ($(c)$) circle (0.3pt) node[above,black] {$\bT$};
\end{tikzpicture}
\caption{\label{fig:barycentres} The group $G_I$ is associated to the barycentre of~$\Delta_I$. The picture shows these groups for $SU(3)$ for all $I$.}
\end{figure}

The groups $\Lambda \cong \pi_1(\bT, 1)$ and $W$ both act on $\liet$ by translation and permuting the coordinates, respectively, which gives rise to an action of the semidirect product $\Waff = \Lambda \rtimes W$, which fits into a split short exact sequence
\begin{equation} \label{eqn:Weyl_group_ext}
	\begin{tikzcd}
		1 \ar[r] & \Lambda \ar[r] & \Waff \ar[r,"q_W",shift left=0.8mm] & W \ar[r] \ar[l,"i_W",shift left=0.8mm] & 1 \ .
	\end{tikzcd}
\end{equation}
Let $\widetilde{W}_I \subset \Waff$ be the stabiliser of $\xi_I \in \liet$. Note that $q_W$ restricts to an isomorphism $q_{W_I} \colon \widetilde{W}_I \to W_I$. Its inverse $\varphi_I \colon W_I \to \widetilde{W}_I$ is defined by
\begin{equation} \label{eqn:inv_of_quot_map}
	\varphi_I(\sigma) = \left(\xi_I - \sigma \cdot \xi_I,\ \sigma\right)\ .
\end{equation}
The map $c_I \colon W_I \to \Lambda$ with $c_I(\sigma) = \xi_I - \sigma \cdot \xi_I$ is the cocycle (which is also a coboundary) for the pullback of the extension \eqref{eqn:Weyl_group_ext} to $W_I$.

For the spectral sequence computing $K_*^G(C^*\cE)$ we need an equivariant closed cover of $G$. For $0 < \delta_\ell < 1$ define 
\begin{align}
	\label{eqn:equiv_cover}
	A_i &= \left\{ (t_1, \dots, t_\ell) \in \Delta^\ell\  | \ t_i \geq 1 - \delta_\ell \right\} \quad \text{for } 1 \leq i \leq \ell\ ,\\
	A_0 &= \left\{ (t_1, \dots, t_\ell) \in \Delta^\ell \ | \ \sum_{j=1}^n t_j \leq \delta_\ell \right\} \notag\ .
\end{align}
In \cite{paper:EvansPennig-Twists} a different parametrisation of $\Delta^\ell$ was used. Apart from this, these are the same sets as in \cite{paper:EvansPennig-Twists}. We can choose $\delta_\ell$ in such a way that $\bigcup_{j=0}^\ell A_j = \Delta^\ell$ (any $\delta_\ell$ with the property $\delta_\ell > 1/(1 + \tfrac{1}{\sqrt{\ell}})$ will work). Let $q \colon G \to \Delta^{\ell}$ be the continuous map sending an element to the point in $\Delta^\ell$ corresponding to its conjugacy class. For each non-empty subset $I \subset \{0, \dots, \ell\}$ let
\begin{equation} \label{eqn:AI_VI}
	A_I = \bigcap_{i \in I} A_i \qquad \text{and} \qquad \hat{A}_I = q^{-1}(A_I) \subseteq G\ .
\end{equation}
Note that $\xi_I \in A_I$, which gives rise to an embedding $\iota_I \colon G/G_I \to \hat{A}_I$. By \cite[Lem.~4.8]{paper:EvansPennig-Twists} this is a $G$-equivariant deformation retract, so in particular a $G$-equivariant homotopy equivalence. 

We now have to consider restrictions of the Fell bundle $\cE$ to closed subsets of $A \subseteq \bT$. Let $Y_A$ be the restriction of $Y \to G$ to $A \subseteq \bT \subset G$, i.e.
\[
	Y_A = \{ (w, z) \in A \times S^1\setminus \{1\}\ |\ z \neq w_i \text{ for all } i \in \{1,...,n\} \}\ ,
\]
where we identify $\bT$ with the subset of $(S^1)^n$ in which the coordinates $w_i$ multiply to $1$. Denote by $Y^{[2]}_A$ the subgroupoid given by the restriction of $Y^{[2]}$, to $A$, i.e.\ 
\begin{align*}
	Y^{[2]}_A &= \{ (w,z_1,z_2) \in Y^{[2]} \ |\ w \in A\} \ .
\end{align*}
Analogously, denote by $\cE_A \to Y_A^{[2]}$ the restriction of $\cE$ to the subgroupoid $Y_{A}^{[2]} \subset Y^{[2]}$. If $A = \{z\}$, then we will write $\cE_z$ instead of $\cE_{\{z\}}$.

Let $w_I = \exp(\xi_I) \in \bT$. A variation of the Mayer-Vietoris spectral sequence for the closed cover $(\hat{A}_i)_{i \in I}$ has the $E_1$-term 
\begin{equation} \label{eqn:spec_seq}
	E_1^{p,q} = \bigoplus_{\lvert I \rvert = p+1} K^{G_I}_q(C^*\cE_{w_I}) \cong 
	\begin{cases}
		\bigoplus_{\lvert I \rvert = p+1} R_F(G_I) & \text{for } q \text{ even}\ , \\
		0 & \text{for } q \text{ odd}
	\end{cases}
\end{equation}
and converges to $K_*^G(C^*\cE)$ by \cite[Prop.~4.9]{paper:EvansPennig-Twists}.

We will identify $R(\bT)$ with a quotient of a polynomial ring using the ring isomorphism that maps $t_i$ to the $i$th projection $\bT \to U(1)$:
\[
	R(\bT) \cong \Z[t_1, \dots, t_n]/(t_1 \cdots t_n -1) \ .
\]
With respect to this isomorphism the restriction of the standard representation $\left.\rho\right|_{\bT}$ corresponds to $t_1 + \dots + t_n$. Therefore
\begin{align*}
	R_F(\bT) &= R(\bT)[F(t_1 + \dots + t_n)^{-1}] \\
	&= R(\bT)[F(t_1)^{-1}, \dots, F(t_n)^{-1}]\ ,
\end{align*}
where the last equality follows from the exponential property which implies that $F(t_1 + \dots + t_n) = F(t_1) \cdots F(t_n)$.

For the convenience of the reader we summarise the definitions that were made above in Table~\ref{tab:symbols}.
\begin{table}[h]
	\renewcommand{\arraystretch}{1.3}
	\begin{tabular}{|c|l|}
		\hline
		Symbol & Description \\
		\hline
		$\bT$ & maximal torus in $G = SU(n)$ given by diag.\ matrices \\
		$\lieg, \liet$ & Lie algebras of $G$ and $\bT$, respectively \\
		$\Lambda$ & integral weight lattice, kernel of $\exp \colon \liet \to \bT$ \\
		$W$ & $\cong S_n$, Weyl group of $G = SU(n)$ \\
		$\Waff$ & $= \Lambda \rtimes W$, affine Weyl group of $G$  \\
		\hline
		$I$ & subset of the vertices of the fundamental alcove \\
		$\Delta_I$ & subsimplex of the fundamental alcove spanned by $I$ \\
		$\xi_I$ & barycentre of $\Delta_I \subset \liet$ \\
		$G_I$ & centraliser of $w_I = \exp(\xi_I)$ in $G$ \\
		$W_I$ & stabiliser of $w_I = \exp(\xi_I)$ in $W$ \\
		\hline
		$Y$ &  $= \{ (g,z) \in G \times S^1 \setminus \{1\}\ |\ z \notin \EV{g}\}$, see \eqref{eqn:Y} \\
		$Y^{[2]}$ & groupoid given by the fibre square of $Y$ over $G$ \\
		$\MF$ & $=\Endo{F(\C^n)}^{\otimes \infty}$, infinite UHF-$G$-$C^*$-algebra\\
		$\cE$ & Fell bundle over $Y^{[2]}$ depending on exp.\ functor $F$ \\
		$Y^{[2]}_A$ & subgroupoid of $Y^{[2]}$ obtained by restriction to $A \subseteq G$ \\
		$\cE_A$ & restriction of $\cE$ to $Y^{[2]}_A$ for $A \subseteq \bT \subset G$ \\
		$Y^{[2]}_{z}, \cE_{z}$ & restrictions of $Y^{[2]}$ and $\cE$, respectively, to $\{z\} \subset \bT$ \\
		\hline
	\end{tabular}
	\renewcommand{\arraystretch}{1}
	\vspace{1mm}
	\caption{\label{tab:symbols} Notation used throughout the paper}
\end{table}

\subsection{Restriction to the maximal torus} \label{subsec:maxtorus}
In this section we will compare the spectral sequence in \eqref{eqn:spec_seq} to another one that computes the $\Waff$-equivariant Bredon cohomology $H_{\Waff}^*(\liet; \cR)$ of $\liet$ with respect to a local coefficient system~$\cR$. As we will see, this comparison is based on the fact that the Fell bundle $\cE$ has a tensor product decomposition when restricted to the maximal torus (see \eqref{eqn:ET-iso} for the precise form of this decomposition). We will start by decomposing the algebra $\MF = \Endo{F(\C^n)}^{\otimes \infty}$. Let $\{e_1, \dots, e_n\}$ be the standard basis of $\C^n$. For $i \in \{1,\dots, n\}$ define
\begin{equation} \label{eqn:MFi}
	\MFfac{i} = \Endo{F(\text{span}\{e_i\})}^{\otimes \infty}
\end{equation}
and let 
\begin{equation} \label{eqn:Vi}
	\cV_i = F(\text{span}\{e_i\}) \otimes \MFfac{i} \ .
\end{equation}
This is an $\MFfac{i}$-$\MFfac{i}$ Morita equivalence bimodule with left and right multiplication analogous to \eqref{eqn:subspace_module}. Since $F$ is a functor, we obtain a unitary $S^1$-representation given by $z\mapsto F(z) \in U(F(\text{span}(e_i)))$. Likewise, each algebra $\MFfac{i}$ carries an $S^1$-action constructed as the infinite tensor product of $z \mapsto \Ad_{F(z)}$. This action extends to the bimodule $\cV_i$ via $z \mapsto F(z) \otimes \Ad_{F(z)}^{\otimes \infty}$ and turns it into a $S^1$-equivariant $\MFfac{i}$-$\MFfac{i}$ Morita equivalence.

The graded tensor product of all endomorphism algebras evaluates to
\[
	\bigotimes_{i=1}^n\Endo{F(\text{span}\{e_i\})} \cong \Endo{F\left(\bigoplus_{i=1}^n \text{span}\{e_i\} \right)} = \Endo{F(\C^n)}
\]
and the tensor product of these isomorphisms gives a $\bT$-equivariant $*$-iso\-mor\-phism of the UHF-algebras 
\begin{equation} \label{eqn:factor_iso}
	\theta \colon \bigotimes_{i=1}^n\MFfac{i} \to \MF\ .
\end{equation}

\begin{remark} \label{rem:symmetric_F}
	This is one of the subtle points where we reap the benefits of symmetric monoidal functors: If $F$ does not preserve the symmetry, then there are several natural maps from the tensor product in the domain to~$\MF$. For example for $n = 2$, the isomorphism
	\[
		\MFfac{1} \otimes \MFfac{2} \to \MFfac{2} \otimes \MFfac{1} \to \Endo{F(\text{span}\{e_2\} \oplus \text{span}\{e_1\})}^{\otimes \infty} \to \MF\ ,
	\]
	where we first interchange the tensor factors, is potentially different from 
	\[
		\MFfac{1} \otimes \MFfac{2} \to \Endo{F(\text{span}\{e_1\} \oplus \text{span}\{e_2\})}^{\otimes \infty} \to \MF\ .
	\]
	If $F$ preserves symmetries, then they agree. As we have seen in Sec.~\ref{sec:equiv_higher_twists} our main examples of exponential functors do in fact preserve symmetries, when considered as functors to super-vector spaces.
\end{remark}

For any $m \in \Z$ and $\cV_i$, $\MFfac{i}$ as in \eqref{eqn:Vi}, \eqref{eqn:MFi}, respectively, we define
\[
	\cV_i^{\otimes m} = 
	\begin{cases}
		\cV_i^{\otimes m} & \text{for } m > 0\ , \\
		\MFfac{i} & \text{for } m = 0 \ , \\
		(\cV_i^{\text{op}})^{\otimes (-m)} & \text{for } m < 0\ ,
	\end{cases}
\]
where all tensor products are taken over $\MFfac{i}$. Note that for all $r,s \in \Z$ we have $S^1$-equivariant bimodule isomorphism 
\begin{equation} \label{eqn:multiplication}
	\cV_i^{\otimes r} \otimes_{\MFfac{i}} \cV_i^{\otimes s} \to \cV_i^{\otimes (r+s)}
\end{equation}
that are associative in the obvious sense. For each $m_1, \dots, m_n \in \Z$ we can turn the bimodule 
\[
	\cV_1^{\otimes m_1} \otimes_{\C} \dots \otimes_{\C} \cV_n^{\otimes m_n}
\]
(note that the tensor products are graded outer tensor products over $\C$ as in \cite[14.4.4]{book:Blackadar}) into an $\MF$-$\MF$-Morita equivalence using~$\theta$ as in \eqref{eqn:factor_iso}. Combining all $S^1$-actions on the $\cV_i$'s we obtain an action of $\bT \subset (S^1)^n$ on this bimodule. For $I \subset \{1,\dots, n\}$ let 
\[
	V_I = \text{span}\{e_i\ |\ i \in I\} \subset \C^n\ .
\] 
Let $m \colon \{1,\dots, n\} \to \{0,1\}$ be the indicator function of $I$. By our observations above we obtain a $\bT$-equivariant bimodule isomorphism
\begin{equation} \label{eqn:bimod_iso}
	\bigotimes_{i = 1}^n \cV_i^{\otimes m(i)} \cong F(V_I) \otimes \MF\ ,
\end{equation}
where the $n$-fold outer tensor product on the left is a graded tensor product over $\C$ and the left hand side is an $\MF$-$\MF$-bimodule via $\theta$.

For $z \in S^1 \subset \C$ denote by $\log_z \colon S^1 \to i\R$ the logarithm with $\log_z(1) = 0$ and cut through $z$ (for example $\log_{-1}$ takes values in $(-\pi i , \pi i) \subset i\R$). Now consider $l_z \colon S^1 \to \R$ with $l_z(w) = \frac{1}{2 \pi i} \log_z(w)$. Let $w,z,x \in S^1$ with $x \neq z$ and $x \neq w$ and note that by \cite[Lem.~5.12]{paper:BeckerMurrayStevenson}
\[
	l_z(x) - l_w(x) = \begin{cases}
		-1 & \text{if } z < x < w\ ,\\
		1 & \text{if } w < x < z\ ,\\
		0 & \text{else}\ .
	\end{cases}
\]
The fibres of $\cE_{A}$ for $A \subseteq \bT$ can be written as follows
\begin{equation} \label{eqn:ET-iso}
	(\cE_{A})_{(w,z_1,z_2)} \cong \bigotimes_{i=1}^n \cV_i^{\otimes (l_{z_2}(w_i) - l_{z_1}(w_i))}\ ,
\end{equation}
where the tensor products on the right hand side are graded tensor products over $\C$. Indeed, for $z_1 < z_2$ the isomorphism \eqref{eqn:bimod_iso} gives in this case
\begin{align*}
	\bigotimes_{i=1}^n \cV_i^{\otimes (l_{z_2}(w_i) - l_{z_1}(w_i))} \cong\ F\!\left( \bigoplus_{z_1 < \lambda < z_2 \atop \lambda \in \EV{w}} \Eig{w}{\lambda} \right) \otimes \MF = (\cE_A)_{(w,z_1,z_2)}\ .
\end{align*}
Since \eqref{eqn:bimod_iso} is $\bT$-equivariant, the above isomorphism is as well. 

Similar to \cite[Sec.~4.1]{paper:BeckerMurrayStevenson} we can now compare $\cE_A$ to another Fell bundle defined as follows: Consider the exponential map $\exp \colon \liet \to \bT$ and let $\liet^{[2]}$ be the fibre product over $\bT$, i.e.\ 
\[
	\liet^{[2]} = \{ (x_1, x_2) \in \liet^2 \ |\ x_1 - x_2 \in \Lambda \}\ .
\]

Let $q_i \colon \Lambda \to \Z$ for $i \in \{1, \dots, n\}$ be the projection map onto the $i$th coordinate of $\Lambda \subset \Z^n$. The connected components of $\liet^{[2]}$ are labelled by $\Lambda$, since $\liet^{[2]} \cong \liet \times \Lambda$. For $\lambda \in \Lambda$ denote the component of $\liet^{[2]}$ by $\liet^{[2]}_{\lambda}$, i.e.\footnote{There is a slight clash of notation here with $\liet^{[2]}_A$ later, but it is fairly clear from the context which is meant.} 
\[
	\liet^{[2]}_{\lambda} = \{ (x_1,x_2) \in \liet^{[2]} \ |\ x_2 - x_1 = \lambda\}\ .
\]
Now consider the following bundle
\begin{equation} \label{eqn:Li}
	\cL_i = \coprod_{\lambda \in \Lambda} \liet^{[2]}_{\lambda} \times \cV_i^{\otimes q_i(\lambda)}
\end{equation}
over $\liet^{[2]}$, where the tensor product on the right hand side is taken over $\MFfac{i}$. The canonical bimodule isomorphisms 
\[
	\cV_i^{\otimes q_i(\lambda)} \otimes_{\MFfac{i}} \cV_i^{\otimes q_i(\mu)} \to \cV_i^{\otimes q_i(\lambda + \mu)}
\] 
turn each $\cL_i$ into a Fell bundle over the groupoid $\liet^{[2]}$. This Fell bundle comes equipped with a fibrewise $S^1$-action induced by the one on $\cV_i$. Let 
\begin{equation} \label{eqn:Ltensor}
	\cL = \cL_1 \otimes_\C \dots \otimes_\C \cL_n
\end{equation}
be the fibrewise outer tensor product of the $\cL_i$'s over $\liet^{[2]}$. It is straightforward to see that this gives a Fell bundle over $\liet^{[2]}$, where the multiplication reshuffles the tensor factors and uses the multiplication in each of the $\cL_i$'s. Combining the $S^1$-actions on the tensor factors and restricting to $\bT \subset (S^1)^n$ we obtain a $\bT$-equivariant Fell bundle over $\liet^{[2]}$.

For a closed subset $A \subseteq \bT$ let $\liet_A = \exp^{-1}(A) \subseteq \liet$ and let $\liet_A^{[2]}$ be the fibre product of $\liet_A$ with itself over $A$. Let $\cL_A \to \liet_A^{[2]}$ be the corresponding restriction of the Fell bundle $\cL$.  

From the identification of the fibres of $\cE_A$ in \eqref{eqn:ET-iso} we see that we have a bimodule between $\cE_A$ and $\cL_A$ constructed as follows: Let $P_A = \liet_A \times_{A} Y_A$ be given by 
	\[
		P_A = \{ (x, w, z) \in \liet \times A \times S^1 \!\setminus\! \{1\}\ | \exp(x) = w, z \neq w_i \text{ for all } i \in \{1, \dots, n\}\}.
	\]
	This space carries a canonical left action of the groupoid $\liet^{[2]}_A$ and a canonical right action of $Y^{[2]}_A$, which turns $P_A$ into a $\liet^{[2]}_A$-$Y^{[2]}_A$-Morita equivalence. Given $(x,w,z) \in P_A$ the condition $\exp(x) = w$ implies that $l_z(w_i) - q_i(x) \in \Z$, where $q_i \colon \liet \to \R$ is the $i$th projection map and $l_z$ denotes (up to a factor) the logarithm with cut at $z$ as above.  Thus, we can consider the (locally trivial) Banach bundle $\cF_A \to P_A$ with fibres defined as follows
	\begin{equation} \label{eqn:FA}
		(\cF_A)_{(x,w,z)} = \bigotimes_{i=1}^n \cV_i^{\otimes (l_z(w_i) - q_i(x))}\ ,
	\end{equation}
	where the $n$-fold outer tensor product is a graded tensor product over $\C$ and the interior tensor product is over $\MFfac{i}$ as above.  For $x, x_1, x_2 \in \liet$, $w \in A$ and $z, z_1, z_2 \in S^1\setminus\{1\}$ such that $(x_i, w, z) \in P_A$ and $(x,w,z_i) \in P_A$ the isomorphisms \eqref{eqn:multiplication} give rise to
	\begin{gather*}
		\cV_i^{\otimes (q_i(x_2) - q_i(x_1))} \otimes_{\MFfac{i}} \cV_i^{\otimes (l_z(w_i) - q_i(x_2))} \to \cV_i^{\otimes (l_z(w_i) - q_i(x_1))} \ ,\\
		\cV_i^{\otimes (l_{z_1}(w_i) - q_i(x))} \otimes_{\MFfac{i}} \cV_i^{\otimes (l_{z_2}(w_i) - l_{z_1}(w_i))} \to \cV_i^{\otimes (l_{z_2}(w_i) - q_i(x))} \ .
	\end{gather*}
	These piece together to give a left action by $\cL_A$ and (using \eqref{eqn:bimod_iso}) a right action by $\cE_A$ on the bundle $\cF_A$:
	\begin{gather}
		(\cL_A)_{(x_1,x_2)} \otimes (\cF_A)_{(x_2,w,z)} \to (\cF_A)_{(x_1,w,z)}	\label{eqn:left-LA-action} \ ,\\
		(\cF_A)_{(x,w,z_1)} \otimes (\cE_{A})_{(w,z_1,z_2)} \to (\cF_A)_{(x,w,z_2)} \label{eqn:right-EA-action}\ .
	\end{gather}
	The associativity of the isomorphisms \eqref{eqn:multiplication} implies $(\ell \cdot f) \cdot e = \ell \cdot (f \cdot e)$ for all $\ell \in (\cL_A)_{(x_1,x_2)}$, $f \in (\cF_A)_{(x_2, w, z_1)}$ and $e \in (\cE_A)_{(w,z_1.z_2)}$. 
	
	Recalling the structure of the opposite bimodule $\cV_i^{\rm op}$ (see \cite[Sec.~2.1]{paper:EvansPennig-Twists}) we obtain graded isomorphisms
	\(
		\left(\cV_i^{\otimes m}\right)^\op \to \cV_i^{\otimes (-m)}
	\).
	From the inner product on $\cV_i^{\otimes m}$ we therefore obtain an antilinear map 
	\[
		\cV_i^{\otimes m} \to \left(\cV_i^{\otimes m}\right)^\op \to \cV_i^{\otimes (-m)} \quad, \quad f \mapsto f^* \ .
	\] 
	It gives rise to two inner products $\lscal{\cL_A}{\,\cdot\,}{\,\cdot\,}$ and $\rscal{\,\cdot\,}{\,\cdot\,}{\cE_A}$ as follows: for $f_1 \in (\cF_A)_{(x_1,w,z)}$, $f_2 \in (\cF_A)_{(x_2,w,z)}$, $g_1 \in (\cF_A)_{(x,w,z_1)}$ and $g_2 \in (\cF_A)_{(x,w,z_2)}$ we define
	\begin{align*}
		\lscal{\cL_A}{f_1}{f_2}\ \ \, &= f_1 \cdot f_2^* \quad \in (\cL_A)_{(x_1,x_2)} \ ,\\
		\rscal{g_1}{g_2}{\cE_A} &= g_1^* \cdot g_2 \quad \in (\cE_A)_{(w,z_1,z_2)}\ ,
	\end{align*}
	where we identify the fibres of $\cE_{A}$ with tensor products of $\cV_i$'s as in \eqref{eqn:ET-iso}. 

\begin{lemma} \label{lem:equiv_mor_eq}
	Let $A \subseteq \bT$ be a closed subset. The Banach bundle $\cF_A \to P_A$ defined above gives rise to a $\bT$-equivariant Morita equivalence between the two Fell bundles $\cL_A \to \liet^{[2]}_A$ and $\cE_A \to Y^{[2]}_A$ in the sense of \cite[Def.~6.1]{paper:MuhlyWilliams}. Consequently,
	\[
		K_*^\bT(C^*\cL_A) \cong K_*^\bT(C^*\cE_A)\ .
	\]
\end{lemma}

\begin{proof}
 The algebraic properties \cite[Def.~6.1, (b) (i) -- (iv)]{paper:MuhlyWilliams} are easily checked. Moreover, each fibre $(\cF_A)_{(x,w,z)}$ is an $(\cL_A)_{(x,x)}$-$(\cE_A)_{(w,z,z)}$ Morita equivalence, since $(\cL_A)_{(x,x)} = \MFfac{1} \otimes \dots \otimes \MFfac{n}$ and $(\cE_A)_{(w,z,z)} = \MF$ and $(\cF_A)_{(x,w,z)}$ is an imprimitivity bimodule for those algebras. 
	
	Therefore $\cF_A$ is an $\cL_A$-$\cE_{A}$-equivalence in the sense of \cite[Def.~6.1]{paper:MuhlyWilliams} and by \cite[Thm.~6.4]{paper:MuhlyWilliams} a completion of the compactly supported sections $C_c(P_A,\cF_A)$ with left and right action and inner products as stated in \cite[Thm.~6.4]{paper:MuhlyWilliams} is an imprimitivity bimodule between $C^*\cL_A$ and $C^*\cE_{A}$. By construction each $\cV_i$ is a graded $\MFfac{i}$-$\MFfac{i}$-Morita equivalence. Therefore the same is true for the completion of $C_c(P_A.\cF_A)$.   
	
	It remains to be seen why it defines a $\bT$-equivariant Morita equivalence. Note that $\bT$ acts trivially on $Y_A^{[2]}$ and $\liet^{[2]}_A$, since this action is induced by restricting the conjugation action. The Banach bundle $\cF_A \to P_A$ carries a fibrewise $\bT$-action induced by the $S^1$-actions on the $\cV_i$'s. The $\bT$-equivariance of the multiplication isomorphisms \eqref{eqn:multiplication} implies that the left action isomorphism \eqref{eqn:left-LA-action} and the right action isomorphism \eqref{eqn:right-EA-action} are both $\bT$-equivariant. The operation $(\,\cdot\,)^*$ intertwines the $S^1$-action on $\cV_i$ with the one on the opposite bimodule $\cV_i^{\text{op}}$. Therefore the $\cL_A$- and $\cE_A$-valued inner products are both $\bT$-equivariant as well. This implies that the completion of $C_c(P_A,\cF_A)$ is a $\bT$-equivariant imprimitivity bimodule between $C^*\cL_A$ and $C^*\cE_A$ in the sense of \cite[Def.~7.2]{book:RaeburnWilliams}.
\end{proof}

The additional definitions from this section can be found in Table~\ref{tab:symbols2} for the convenience of the reader.
\begin{table}[h]
	\renewcommand{\arraystretch}{1.3}
	\begin{tabular}{|c|l|}
		\hline
		Symbol & Description \\
		\hline
		$\liet^{[2]}$ & groupoid given by the fibre square of $\liet$ over $\bT$ \\
		$\cL_i$ & Fell bundle over $\liet^{[2]}$ with fibre $\cV_i$, see \eqref{eqn:Vi} \\
		$\cL$ & Fell bundle over $\liet^{[2]}$ given by $\cL_1 \otimes \dots \otimes \cL_n$ \\
		$\cL_A$ &  restriction of $\cL$ to $A \subseteq \bT \subset G$\\
		$\cF_A$ & Morita equivalence bundle between $\cE_A$ and $\cL_A$\\
		\hline
	\end{tabular}
	\renewcommand{\arraystretch}{1}
	\vspace{1mm}
	\caption{\label{tab:symbols2} Notation used throughout the paper}
\end{table}

\subsubsection{Normaliser and Weyl group actions} \label{subsec:WeylGroup}
Let $I \subset \{0, \dots, \ell\}$ be a subset of the vertices of $\Delta^\ell$ (where $\ell = n-1$ is the rank of $G$). Recall that $G_I$ is the centraliser of $w_I = \exp(\xi_I)$, see \eqref{eqn:GI-iso}. It is connected by \cite[part E, Ch.~II, Thm.~3.9]{book:SeminarAlgGroups}. Moreover, the Weyl group of $G_I$ is $W_I$. This implies that the restriction homomorphism $R(G_I) \to R(\bT)^{W_I}$ is an isomorphism. The element $F(\left.\rho\right|_{\bT})$ for the standard representation $\rho$ is invariant under $W_I$. Therefore this isomorphism survives the localisation and we have  
\[
\begin{tikzcd}
	R_F(G_I) \ar[r,"\cong"] & R_F(\bT)^{W_I}
\end{tikzcd}
\]
induced by the restriction map.

Let $N_G(\bT) \subseteq SU(n)$ be the normaliser of the maximal torus. This group consists of ``generalised permutation matrices``, i.e.\ matrices of determinant~1 whose only non-zero entries are complex numbers of norm $1$ that occur exactly once per row and column. It fits into a short exact sequence
\begin{equation} \label{eqn:TNW-sequence}
	1 \to \bT \to N_G(\bT) \to W \to 1\ .
\end{equation}
The group $W \cong S_n$ in this sequence is the Weyl group, which acts on $\bT \subset (S^1)^n$ by permuting the coordinates. This lifts to a corresponding action of $W$ on $Y_\bT$ with $\sigma \cdot (w,z) = (\sigma \cdot w, z)$ for $(w,z) \in Y_\bT$ and $\sigma \in W$. Hence, $W$ also acts diagonally on $Y_\bT^{[2]}$ by groupoid isomorphisms. Each element $\hat{\sigma} \in N_G(\bT)$ gives rise to a unitary transformation
\(
	\hat{\sigma} \colon \C^n \to \C^n 
\).

For a given point $(w,z_1,z_2) \in Y_\bT^{[2]}$ with $z_1 < z_2$ the map $\hat{\sigma}$ that lifts $\sigma \in W$ restricts to a unitary isomorphism of eigenspaces (note that $\sigma \cdot w = \hat{\sigma} w \hat{\sigma}^*$) 
\[
	\hat{\sigma}_{(w,z_1,z_2)} \colon \bigoplus_{z_1 < \lambda < z_2 \atop \lambda \in \EV{w}} \Eig{w}{\lambda} \to \bigoplus_{z_1 < \lambda < z_2 \atop \lambda \in \EV{\sigma \cdot w}} \Eig{\sigma \cdot w}{\lambda}
\]
and likewise for $z_1 \geq z_2$. We can consider the $C^*$-algebra $\MF$ (see Tab.~\ref{tab:symbols}) as an $N_G(\bT)$-algebra with the action given by $(\Ad_{F(\hat{\sigma})})^{\otimes \infty}$. The map 
\[
	F(\hat{\sigma}_{(w,z_1,z_2)}) \otimes \Ad_{F(\hat{\sigma})}^{\otimes \infty}
\] 
induces an isomorphism $(\cE_\bT)_{(w,z_1,z_2)} \to (\cE_{\bT})_{(\sigma \cdot w, z_1,z_2)}$, which intertwines the ordinary left and right $\MF$-action on $(\cE_\bT)_{(w,z_1,z_2)}$ with the ones on $(\cE_{\bT})_{(\sigma \cdot w, z_1,z_2)}$ that are twisted by $(\Ad_{F(\hat{\sigma})})^{\otimes \infty}$ and the $\bT$-action on the domain with the $\sigma$-permuted $\bT$-action on the codomain. 

Altogether, $\cE_\bT$ is an $N_G(\bT)$-equivariant Fell bundle with respect to this action. In fact, what we have described above is just the restriction of the given $G$-action on $\cE$ to an $N_G(\bT)$-action on $\cE_\bT$.

The group $N_G(\bT)$ also acts on $\cL$, defined in \eqref{eqn:Ltensor} and \eqref{eqn:Li}, in the following way: An element $\sigma \in W$ acts on $(x_1,x_2) \in \liet^{[2]}$ by $(\sigma\cdot x_1, \sigma \cdot x_2)$ with $\sigma$ permuting the coordinates of $x_i \in \liet \subset \R^n$. With respect to this action $\exp \colon \liet \to \bT$ is $W$-equivariant. Since $F$ is exponential, the fibres of $\cL$ are bimodules isomorphic to 
\[
	\cL_{(x_1,x_2)} \cong F\left( \bigoplus_{i=1}^n V_i^{\oplus q_i(x_2 - x_1)} \right) \otimes \MF\ .	
\]
where $V_i = \text{span}\{e_i\}$ and we define $V_i^{\oplus m} = (V_i^*)^{\oplus (-m)}$ if $m < 0$ and $V^{\oplus 0} = 0$. An element $\hat{\sigma} \in N_G(\bT)$ lifting $\sigma \in W$ provides a unitary isomorphism
\[	
	\hat{\sigma}_{(x_1,x_2)} \colon \bigoplus_{i=1}^n V_i^{\oplus q_i(x_2 - x_1)} \to \bigoplus_{i=1}^n V_i^{\oplus q_{i}(\sigma\cdot(x_2 - x_1))}
\]
by applying the restriction $\hat{\sigma} \colon V_i \to V_{\sigma(i)}$ or $(\hat{\sigma}^*)^{-1} \colon V_i^* \to V_{\sigma(i)}^*$ to each non-trivial summand. As above, $F(\hat{\sigma}_{(x_1,x_2)}) \otimes (\Ad_{F(\hat{\sigma})})^{\otimes \infty}$ gives an isomorphism between $\cL_{(x_1,x_2)}$ and 
\[
	\cL_{(\sigma \cdot x_1,\sigma \cdot x_2)} \cong 	F\left( \bigoplus_{i=1}^n V_i^{\oplus q_i(\sigma \cdot (x_2 - x_1))} \right) \otimes \MF\ .	
\]
intertwining the ordinary and twisted actions on these bimodules. Altogether, we have turned $\cL$ into an $N_G(\bT)$-equivariant Fell bundle.

Both of the $N_G(\bT)$-actions induce corresponding $W$-actions on the $\bT$-equivariant $K$-groups by the following general observation: Let $D$ be a unital $N_G(\bT)$-$C^*$-algebra. Denote the $N_G(\bT)$-action on $D$ by $\alpha$ and the action of $W$ on~$\bT$ by $\gamma$. Let $(E,\lambda)$ be a finitely generated projective $(\bT, D, \alpha)$-module in the sense of~\cite[Def.~2.2.1]{book:PhillipsEquivariant} on which the $\bT$-action is given by restricting $\alpha$, i.e.\ $E$ is a finitely generated projective right Hilbert $D$-module and $\lambda \colon \bT \to \mathcal{L}(E)$ is a continuous representation such that $\lambda_g(\xi \cdot a) = \lambda_g(\xi) \cdot \alpha_g(a)$. Let $\sigma \in W$, choose a lift $\hat{\sigma} \in N_G(\bT)$ of $\sigma$ and define $E_{\hat{\sigma}}$ to be the same Banach space as~$E$, but with the right $D$-multiplication modified by $\alpha$ as follows:
\[
	v \ast a = v\cdot \alpha_{\hat{\sigma}}(a)
\]
for $v \in E$, $a \in D$. The inner product can be adjusted accordingly. Let $\lambda_{\sigma} \colon \bT \to \mathcal{L}(E)$ be given by $\lambda_\sigma(w) = \lambda(\gamma_\sigma(w))$ for $w \in \bT$. The pair $(E_{\hat{\sigma}}, \lambda_\sigma)$ is again a finitely generated projective $(\bT, D, \alpha)$-module. 

The isomorphism class of $E_{\hat{\sigma}}$ does not depend on the chosen lift $\hat{\sigma}$. Indeed, any two choices $\hat{\sigma}_1, \hat{\sigma}_2$ of a lift of $\sigma$ will differ by an element of $\hat{w} \in \bT$. But $\bT$ is path-connected. Thus, a path between $\hat{w}$ and the identity gives rise to a homotopy between $(E_{\hat{\sigma}_1}, \lambda_{\sigma})$ and $(E_{\hat{\sigma}_2}, \lambda_{\sigma})$, which therefore represent the same element in $K_0^\bT(D)$. Hence, for $[E,\lambda] \in K_0^\bT(D)$ and $\sigma \in W$ 
\[
	\sigma \cdot [E,\lambda] = [E_{\hat{\sigma}^{-1}}, \lambda_{\sigma^{-1}}]
\]
defines a (left) action of $W$ on $K_0^\bT(D)$. Replacing $D$ by $D \otimes \C \ell_1$, where $\C \ell_1$ is the Clifford algebra of $\R$ we see that the $W$-action extends to $K_1^\bT(D)$. 

Finally, we also have a $W$-action on $P_\bT$, the base space of the bimodule bundle $\cF_\bT$ defined in \eqref{eqn:FA}: Let $\sigma \in W$ and $(x,w,z) \in P_\bT$. We define
\[
	\sigma \cdot (x,w,z) = (\sigma \cdot x, \sigma \cdot w, z)\ .
\]
With this action $P_\bT$ turns into a $W$-equivariant Morita equivalence between $\liet^{[2]}$ and $Y_\bT^{[2]}$. This $W$-action lifts to an $N_G(\bT)$-action $\cF_\bT$. Since the fibres of $\cF_\bT$ and $\cL$ are both constructed from the same bimodules $\cV_i$, this action is defined completely analogous to the one on $\cL$ and gives fibrewise $N_G(\bT)$-equivariant isomorphisms
\[
	(\cF_\bT)_{(x,w,z)} \to (\cF_{\bT})_{(\sigma \cdot x, \sigma \cdot w, z)}\ .
\]

Let $A \subseteq \bT$ be a closed subset and let $W_A \subseteq W$ be a subgroup such that $W_A \cdot A = A$. Let $\sfX_A$ be the completion of $C_c(P_A, \cF_A)$. By our observations above it provides a $N_G^A(\bT)$-equivariant imprimitivity bimodule between $C^*\cE_A$ and $C^*\cL_A$, where $N_G^A(\bT)$ is the preimage of $W_A$ in $N_G(\bT)$. Denote the $N_G^A(\bT)$-action on $\sfX_A$ by $\delta$. Let $(E,\lambda)$ be a finitely generated projective $(\bT, C^*\cE_A, \alpha)$-module. For $\sigma \in W_A$ and a lift $\hat{\sigma} \in N_G^A(\bT)$ the map 
\[
	E_{\hat{\sigma}} \otimes_{C^*\cE_A} \sfX_A \to (E \otimes_{C^*\cE_A} \sfX_A)_{\hat{\sigma}} \quad , \quad v \otimes x \mapsto v \otimes \delta_{\hat{\sigma}}(x)
\]
is an isomorphism of Hilbert $C^*\cL_A$-modules intertwining the two $\bT$-actions $\lambda_\sigma \otimes \id{\sfX_A}$ and $(\lambda \otimes \id{\sfX_A})_{\sigma}$. In particular, the following identity holds for classes in $K_0^\bT(C^*\cL_A)$
\[
	[E_{\hat{\sigma}} \otimes_{C^*\cE_A} \sfX_A, \lambda_\sigma \otimes \id{\sfX_A}] = [(E \otimes_{C^*\cE_A} \sfX_A)_{\hat{\sigma}}, (\lambda \otimes \id{\sfX_A})_{\sigma}] \in K_0^\bT(C^*\cL_A)\ .
\]
By forming the tensor product of $\cE_A$, $\cL_A$ and $\cF_A$ with $\C\ell_1$ we may extend this identity to $K_1^\bT(C^*\cL_A)$. Hence, we have proven the following lemma:

\begin{lemma} \label{lem:WTequiv_mor_eq}
	Let $A \subseteq \bT$ be a closed subset and let $W_A \subset W$ be a subgroup such that $W_A \cdot A = A$. The Banach bundle $\cF_A \to P_A$ gives rise to a $N_G^A(\bT)$-equivariant Morita equivalence between $\cL_A \to \liet^{[2]}_A$ and $\cE_A \to Y^{[2]}_A$. This Morita equivalence induces a $W_A$-equivariant isomorphism 
	\[
		K_*^\bT(C^*\cL_A) \cong K_*^\bT(C^*\cE_A)\ .
	\]
\end{lemma}

\subsubsection{Bredon cohomology} \label{subsec:Bredon}
Recall that $\Waff = \Lambda \rtimes W$. Identifying $R_F(G)$ with $R_F(\bT)^W \subseteq R_F(\bT)$ we may consider $R_F(\bT)$ as an $R_F(G)$-module. The exponential functor $F$ induces a group homomorphism 
\begin{align} \label{eqn:hom_to_GL1}
	\psi \colon \Lambda \to GL_1(R_F(\bT)) \quad , \quad (k_1, \dots, k_n) \mapsto &\ F(t_1)^{k_1} \cdots F(t_n)^{k_n} \\
	=&\ F(k_1t_1 + \dots + k_nt_n)\ . \notag
\end{align}
The group $W \cong S_n$ acts on $R_F(\bT)$ by permuting the variables $t_1, \dots, t_n$. Denote this action by $\ast$. The lattice $\Lambda$ acts by multiplication by the element of $GL_1(R_F(\bT))$ corresponding to it under $\psi$. This gives rise to a $\Waff$-action on $R_F(\bT)$ by $R_F(G)$-module isomorphisms defined for $(k,\sigma) \in \Waff$ acting on $f \in R_F(\bT)$ as follows
\begin{equation} \label{eqn:affine_action}
	(k, \sigma) \cdot f = \psi(k)\, (\sigma \ast f)\ .
\end{equation}

Let $\text{Orb}_{\Waff}$ be the orbit category of $\Waff$. Its objects are the sets $\Waff/H$ for subgroups $H \subset \Waff$. Morphisms $\Waff/H_1 \to \Waff/H_2$ are given by $\Waff$-equivariant maps. Such morphisms are in bijection with elements $[x] \in \Waff/H_2$ such that $H_1 \subseteq xH_2 x^{-1}$. A local coefficient system is a contravariant functor 
\[
	\text{Orb}_{\Waff} \to \text{Ab}\ .
\]
Define
\begin{equation} \label{eqn:coeff_sys}
	\cR(\Waff/H) = R_F(\bT)^H \qquad \text{and} \qquad \cR_\Q(\Waff/H) = R_F(\bT)^H \otimes \Q\ .
\end{equation}
A morphism given by $[x] \in \Waff/H_2$ maps an element $f \in R_F(\bT)^{H_2}$ to $x \cdot f \in R_F(\bT)^{H_1}$, where the dot denotes the $\Waff$-module structure from~\eqref{eqn:affine_action}. With this definition $\cR$ and $\cR_\Q$ are local coefficient systems.  

The simplex $\Delta^\ell \subset \liet$ is a fundamental domain for the action of $\Waff$ on $\liet$ and turns this space into a $\Waff$-CW-complex (see Tab.~\ref{tab:symbols} for the notation). Its $k$-cells are labelled by the subsets $I \subset \{0, \dots, \ell\}$ with $\lvert I \rvert = k+1$. Let $\widetilde{q} \colon \liet \to \Delta^\ell$ be the composition of the covering map $\exp \colon \liet \to \bT$ with the quotient map $\bT \to \Delta^\ell$ that parametrises conjugacy classes. From the closed cover $A_i$ of $\Delta^\ell$ defined in \eqref{eqn:equiv_cover} we obtain a closed cover $(B_i)_{i \in \{0, \dots, \ell\}}$ of $\liet$ with $B_i = \widetilde{q}^{-1}(A_i)$. A picture of the cover $(B_i)_{i \in \{0,1,2\}}$ for $SU(3)$ can be found in \cite[Fig.~5]{paper:EvansPennig-Twists}. Let $\xi_I$ be the barycentre of the subsimplex $\Delta_I \subseteq \Delta^\ell$. The cover $(B_i)_{i \in \{0,\dots,\ell\}}$ is $\Waff$-invariant and has the property that the inclusion maps 
\[
	\Waff\cdot \xi_I \to B_I
\] 
are equivariant homotopy equivalences, where $B_I = \bigcap_{i \in I} B_i$. These observations allow us to compute the Bredon cohomology groups $H^k_{\Waff}(\liet, \cR)$ using the Mayer-Vietoris spectral sequence with $E^1$-term
\[
	E^1_{p,q} = \bigoplus_{I \subset \{0,\dots, \ell\} \atop \lvert I \rvert = p+1} H^q_{\Waff}(B_I; \cR)\ .
\]
The inclusions $\Waff\cdot \xi_I \to B_I$ give rise to isomorphisms
\[
	H^q_{\Waff}(B_I; \cR) \cong H^q_{\Waff}(\Waff/\widetilde{W}_I; \cR) \cong
	\begin{cases}
		R_F(\bT)^{\widetilde{W}_I} & \text{if } q = 0 \ ,\\
		0 & \text{else}\ .	
	\end{cases} 
\]
Note that the stabiliser subgroup $\widetilde{W}_I$ is also the stabiliser of any other $\xi$ in the interior of the subsimplex $\Delta_I$. Moreover, for $J \subseteq I$ we have $\widetilde{W}_I \subseteq \widetilde{W}_J$, i.e.\ the stabilisers of points on the bounding faces contain the stabilisers of the interior points. For $J \subset I$ the above isomorphism intertwines the restriction homomorphism $H^0_{\Waff}(	B_J;\cR) \to H^0_{\Waff}(B_I;\cR)$ with the inclusion $R_F(\bT)^{\widetilde{W}_J} \to R_F(\bT)^{\widetilde{W}_I}$. Thus, the $E^1$-page boils down to the cochain complex
\[
	C^k_{\Waff}(\liet; \cR) = \bigoplus_{\lvert I \rvert = k+1} R_F(\bT)^{\widetilde{W}_I} \quad, \quad d_k^{\text{cell}} \colon C^k_{\Waff}(\liet; \cR) \to C^{k+1}_{\Waff}(\liet; \cR)
\]
with differentials given by alternating sums of restriction homomorphisms.

Let $w_I = \exp(\xi_I) \in \bT$ and let $Y_{w_I} = \pi^{-1}(w_I)$ where $\pi \colon Y \to G$ for $Y$ as in \eqref{eqn:Y} is the projection map. We can identify $Y_{w_I}$ with $S^1 \setminus (\{1\} \cup \EV{w_I})$. Let $\cE_{w_I} \to Y^{[2]}_{w_I}$ be the restriction of $\cE$ to the subgroupoid $Y^{[2]}_{w_I}$ of $Y^{[2]}$, which we will identify with 
\[
	 \left\{ (z_1,z_2) \in S^1 \setminus \{1\} \ | \ z_i \notin \EV{w_I} \text{ for } i \in \{1,2\} \right\}\ .
\]
To compare the differentials in the cochain complex $C^*_{\Waff}(\liet;\cR)$ with corresponding homomorphisms in $K$-theory, we need to find explicit isomorphisms $K_*^{G_I}(C^*\cE_{w_I}) \cong K_*^{G_I}(\MF)$. These are given by Morita equivalences that are constructed as in \cite[Lem.~4.3]{paper:EvansPennig-Twists}, which we briefly recall now: Note that any 
\(
	z_0 \in Y_{w_I}
\)
gives a map $\sigma^Y_I \colon Y_{w_I} \to Y^{[2]}_{w_I}$ with $\sigma^Y_I(z) = (z,z_0)$. Let $\tilde{\cF}_{w_I} = (\sigma_I^Y)^*\cE_{w_I}$. It was shown in \cite[Lem.~4.3]{paper:EvansPennig-Twists} that $\tilde{\cF}_{w_I}$ is a $G_I$-equivariant Morita equivalence of Fell bundles between $\cE_{w_I}$ and the trivial bundle over the point with fibre $\MF$. In particular, there are two fibrewise inner products on $\tilde{\cF}_{w_I}$, one with values in $\cE_{w_I}$, the other one with values in $\MF$. The two completions with respect to the norms obtained from them agree and 
\[
	\sfX_{w_I} = \overline{C_c(Y_{w_I}, \tilde{\cF}_{w_I})}^{\lVert \cdot \rVert}
\] 
is a $G_I$-equivariant Morita equivalence bimodule between $C^*(\cE_{w_I})$ acting from the left on~$\sfX_{w_I}$ and $\MF$ acting from the right. We will first use this to show that the restriction map $K_0^{G_I}(C^*\cE_{w_I}) \to K_0^\bT(C^*\cE_{w_I})$ is injective with image equal to the $W_I$-fixed points.

Combining the equivalence $\sfX_{w_I}$ with $K_0^{G_I}(\MF) \cong R_F(G_I)$ (induced by the colimit of the isomorphisms $K_0^{G_I}(\Endo{V}) \cong R(G_I)$) gives 
\[
	K_0^{G_I}(C^*\cE_{w_I}) \cong R_F(G_I)\ .
\] 
As a consequence we obtain the following commutative diagram: 
\[
	\begin{tikzcd}[column sep=1.5cm, row sep=0.6cm]
		K_0^{G_I}(C^*\cE_{w_I}) \ar[r,"\text{res}_\bT^{G_I}"] \ar[d,"\cong" left] & K_0^\bT(C^*\cE_{w_I})^{W_I} \ar[d,"\cong"] \\
		R_F(G_I) \ar[r,"\cong" below] & R_F(\bT)^{W_I}
	\end{tikzcd}
\]
In particular, the map $\text{res}_\bT^{G_I}$ induced by restricting the group action from~$G_I$ to $\bT$ has image in the fixed-points and is an isomorphism. 

The vertical maps in the above diagram depend on the choice of $z_0$, which is difficult to track. Luckily, the Morita equivalence with $C^*\cL_{w_I}$ provides an alternative as follows: Let $\sigma^\liet_I \colon \xi_I + \Lambda \to \liet^{[2]}$ be given by $\sigma^\liet_I(\eta) = (\eta, \xi_I)$. Using \cite[Lem.~4.3]{paper:EvansPennig-Twists} again there is a completion 
\[	
	\sfX_{w_I}^\liet = \overline{C_c(\xi_I + \Lambda, (\sigma_I^\liet)^*\cL)}^{\lVert \cdot \rVert}
\]
of the compactly supported sections, which provides a $\bT$-equivariant Morita equivalence between $C^*\cL_{w_I}$ and $\MF$, and therefore an isomorphism
\[
	K_0^\bT(C^*\cL_{w_I}) \cong K_0^\bT(\MF) \cong R_F(\bT)\ .
\]
We call this Morita equivalence the trivialisation of $C^*\cL_{w_I}$. The $\bT$-equivariant Morita equivalence between $C^*(\cL_{w_I})$ and $C^*(\cE_{w_I})$ constructed in the proof of Lem.~\ref{lem:equiv_mor_eq} combined with the above map induces a group isomorphism  
\[	
	\begin{tikzcd}
		\kappa_I \colon K_0^\bT(C^*\cE_{w_I}) \ar[r,"\cong" above] & K_0^\bT(C^*\cL_{w_I}) \ar[r,"\cong"] & R_F(\bT)
	\end{tikzcd}
\]

By Lem.~\ref{lem:WTequiv_mor_eq} the group $W_I$ acts on the two $K$-groups $K_0^\bT(C^*\cE_{w_I})$ and $K_0^\bT(C^*\cL_{w_I})$ in such a way that the above isomorphism is $W_I$-equivariant. The price to pay for our more natural choice of isomorphism is that this map $K_0^\bT(C^*\cL_{w_I}) \to R_F(\bT)$ will no longer be equivariant with respect to the permutation action of $W_I$ on $R_F(\bT)$ as the next lemma shows. 

\begin{lemma} \label{lem:cochain_complex}
	The isomorphism $\kappa_I \colon K_0^\bT(C^*\cE_{w_I}) \to R_F(\bT)$ satisfies 
	\[
		\kappa_I(\rho \cdot x) = \varphi_I(\rho) \cdot \kappa_I(x)
	\]
	for $x \in K_0^\bT(C^*\cE_{w_I})$ and $\rho \in W_I$, where $\varphi_I \colon W_I \to \widetilde{W}_I$ is the group isomorphism defined in \eqref{eqn:inv_of_quot_map} and the $\widetilde{W}_I$-action is the restriction of \eqref{eqn:affine_action} to $\widetilde{W}_I \subseteq \Waff$. In particular, $\kappa_I$ restricts to an isomorphism 
	\[
		K_0^{G_I}(C^*\cE_{w_I}) \to K_0^\bT(C^*\cE_{w_I})^{W_I} \to R_F(\bT)^{\widetilde{W}_I}
	\]
	that makes the following diagram commute
	\[
		\begin{tikzcd}[column sep=2.5cm, row sep=0.8cm]
			\bigoplus_{\lvert I \rvert = p+1} K_0^{G_I}(C^*\cE_{w_I}) \ar[r,"d_1"] \ar[d,"\bigoplus \kappa_I" left, "\cong" right] & \bigoplus_{\lvert I \rvert = p+2} K_0^{G_I}(C^*\cE_{w_I}) \ar[d,"\bigoplus \kappa_I", "\cong" left] \\
			C^p_{\Waff}(\liet; \cR) \ar[r,"d_1^{\text{cell}}" below] & C^{p+1}_{\Waff}(\liet; \cR)
		\end{tikzcd}
	\]
\end{lemma}

\begin{proof}
	By Lem.~\ref{lem:WTequiv_mor_eq} the isomorphism $K_0^\bT(C^*\cE_{w_I}) \to K_0^\bT(C^*\cL_{w_I})$ induced by $\cF_{w_I}$ is $W_I$-equivariant. Hence, it suffices to consider the $W_I$-equivariance of $K_0^\bT(C^*\cL_{w_I}) \to K_0^\bT(\MF)$. Let $N_I = N_G^{\{w_I\}}(\bT)$ be the preimage of $W_I$ in $N_G(\bT)$ with respect to the quotient map $N_G(\bT) \to W$. Recall that $w_I = \exp(\xi_I)$ is fixed by $W_I$. As a first step we can ``destabilise'' $C^*\cL_{w_I}$ in an $N_I$-equivariant way as follows: Consider the finite set
	\[
		\mathfrak{s}_I = \{ \sigma \cdot \xi_I \in \liet\ |\ \sigma \in W_I \}\ .
	\]
	Let $Q_{w_I} = \{ (x_1,x_2) \in \liet \times \mathfrak{s}_I \ | \ \exp(x_1) = \exp(x_2) = w_I\}$. Let $\iota_I \colon Q_{w_I} \to \liet^{[2]}$ be the inclusion map and let $\hat{\cF}_{w_I} = \iota_I^*\cL$. Likewise, let $j_I \colon \mathfrak{s}_I^2 \to \liet^{[2]}$ be the inclusion of the product into the fibre product. Similar to \cite[Ex.~6.6]{paper:MuhlyWilliams} the Banach bundle $\hat{\cF}_{w_I}$ provides a Morita equivalence between  $\cL_{w_I}$ and the Fell bundle $\hat{\cL}_{w_I} = j_I^*\cL \to \mathfrak{s}_I^2$. By construction this equivalence is $N_I$-equivariant. As explained in Sec.~\ref{subsec:WeylGroup}, the completion of $C_c(Q_{w_I}, \hat{\cF}_{w_I})$ gives a $W_I$-equivariant isomorphism $K_0^\bT(C^*\cL_{w_I}) \to K_0^\bT(C^*\hat{\cL}_{w_I})$. The trivialisation $K_0^\bT(C^*\cL_{w_I}) \to K_0^\bT(\MF)$ factors through $K_0^\bT(C^*\hat{\cL}_{w_I})$. Let
	\[
		\iota_{\mathfrak{s}_I} \colon \mathfrak{s}_I \to \mathfrak{s}_I^2 \quad , \quad \eta \mapsto (\eta, \xi_I)\ .
	\]
	The isomorphism $K_0^\bT(C^*\hat{\cL}_{w_I}) \to K_0^\bT(\MF)$ is induced by the Banach bundle $\mathcal{T}_I \to \mathfrak{s}_I$ given by $\iota_{\mathfrak{s}_I}^*\hat{\cL}$ with fibres
	\[
		(\mathcal{T}_I)_{\sigma \cdot \xi_I} = \bigotimes_{j=1}^n \cV_j^{\otimes q_j(\xi_I - \sigma \cdot \xi_I)} \cong F\!\left( \bigoplus_{j=1}^n V_j^{\oplus q_j(\xi_I - \sigma \cdot \xi_I)}\right) \otimes \MF\ .
	\]
	where $V_j = \text{span}\{e_j\}$. Analogous to the proof of Lem.~\ref{lem:equiv_mor_eq} the isomorphisms
	\[
		\cV_j^{\otimes q_j(\sigma_2 \cdot \xi_I - \sigma_1 \cdot \xi_I)} \otimes_{\MF} \cV_j^{\otimes q_j(\xi_I - \sigma_2 \cdot\xi_I)} \to \cV_j^{\otimes q_j(\xi_I - \sigma_1 \cdot\xi_I)} 
	\]
	combine to define a left action of $\hat{\cL}_{w_I}$ on $\mathcal{T}_I$. Together with the canonical right action by $\MF$ the bundle $\mathcal{T}_I$ provides a Morita equivalence between $\hat{\cL}_{w_I}$ and the trivial Fell bundle with fibre $\MF$ over the point. Let
	\[
		\sfX_I = C(\mathfrak{s}_I, \mathcal{T}_I) = \bigoplus_{\sigma \in W_I} \bigotimes_{j=1}^n \cV_j^{\otimes q_j(\xi_I - \sigma \cdot \xi_I)}
	\]
	be the associated $C^*\hat{\cL}_{w_I}$-$\MF$-imprimitivity bimodule. The Banach bundle $\mathcal{T}_I$ carries a fibrewise $\bT$-action that turns $\sfX_I$ into a $\bT$-equivariant bimodule. However, the homomorphism on $K_0^\bT$ induced by $\sfX_I$ is not $W_I$-equivariant when $K_0^\bT(\MF)$ is equipped with the $W_I$-action induced by the natural $N_I$-action on the algebra. Let $\hat{\rho} \in N_I$ be a lift of $\rho \in W_I$. This lift induces a unitary isomorphism
	\[
		\bigoplus_{j=1}^n V_j^{\oplus q_j(\xi_I - \sigma \cdot \xi_I)} \to \bigoplus_{j=1}^n V_j^{\oplus q_j(\rho \cdot \xi_I - \rho\cdot \sigma \cdot \xi_I)}
	\]
	as described in Sec.~\ref{subsec:WeylGroup}. By applying $F$ to it and taking the direct sum over all $\sigma \in W_I$ we obtain a bimodule isomorphism 
	\[
		\sfX_I \to 
		\tensor*[_{\hat{\rho}\!\!}]{\left(\bigoplus_{\sigma \in W_I} \bigotimes_{j=1}^n \cV_j^{\otimes q_j(\rho \cdot \xi_I - \rho \cdot \sigma \cdot \xi_I)}\right)}{_{\!\!\!\hat{\rho}}}
		\cong \tensor*[_{\hat{\rho}\!\!}]{\left(\bigoplus_{\sigma \in W_I} \bigotimes_{j=1}^n \cV_j^{\otimes q_j(\rho \cdot \xi_I - \sigma \cdot \xi_I)}\right)}{_{\!\!\!\hat{\rho}}}\ ,
	\]
	where the subscript $\hat{\rho}$ denotes the $(\Ad_{F(\hat{\rho})})^{\otimes \infty}$-twisted left and right actions of $C^*\hat{\cL}_{w_I}$ and $\MF$, respectively. The codomain of the above isomorphism is isomorphic as a bimodule to: 
	\[
		\tensor*[_{\hat{\rho}\!\!}]{\left(\sfX_I \otimes_{\MF}  \bigotimes_{j=1}^n \cV_j^{\otimes q_j(\rho \cdot \xi_I - \xi_I)}\right)}{_{\!\!\!\hat{\rho}}} \cong \tensor*[_{\hat{\rho}}]{(\sfX_I)}{_{\hat{\rho}}} \otimes_{\MF} \bigotimes_{j=1}^n \cV_j^{\otimes q_j(\xi_I - \rho^{-1}\cdot\xi_I)} \ ,
	\]
	where we applied the map induced by $\hat{\rho}^{-1}$ to the second tensor factor.	
	
	Let $(E,\lambda)$ be a finitely generated projective $(\bT, \MF, \alpha)$-module and let $\rho \in W_I$. Let $x = \kappa_I^{-1}([E,\lambda])$. The element $\kappa_I( \rho \cdot x) \in K_0^\bT(\MF)$
	is represented by the module $(E \otimes_{\MF} \sfX_I^{\text{op}})_{\hat{\rho}^{-1}} \otimes_{C^*\hat{\cL}} \sfX_I$ for an arbitrary lift $\hat{\rho} \in N_I$ of $\rho \in W_I$. (The inverse $\hat{\rho}^{-1}$ appears here, because we wanted the action of $W_I$ to be a left action.) A brief computation shows 
	\begin{align*} 
		 \ \ &(E \otimes_{\MF} \sfX_I^{\text{op}})_{\hat{\rho}^{-1}} \otimes_{C^*\hat{\cL}} \sfX_I \\  
	\cong  &\left(E \otimes_{\MF} (\sfX_I^{\text{op}} \otimes_{C^*\hat{\cL}} \sfX_I) \right)_{\!\!\hat{\rho}^{-1}} \otimes_{\MF} \bigotimes_{j=1}^n \cV_j^{\otimes q_j( \xi_I - \rho \cdot\xi_I)}\\[-4mm]
	\cong &\ E_{\hat{\rho}^{-1}} \otimes_{\MF} \bigotimes_{j=1}^n \cV_j^{\otimes q_j(\xi_I - \rho \cdot \xi_I)}\ .
	\end{align*}
	After applying the isomorphism $K_0^\bT(\MF) \cong R_F(\bT)$ the module in the last line represents the $K$-theory class 
	\begin{align*} 
		 F\!\left(\sum_{j=1}^n q_j( \xi_I - \rho\cdot \xi_I)t_j\right) \! (\rho \ast \kappa_I(x)) 
		= \psi(c_I(\rho))\, (\rho \ast \kappa_I(x)) = \varphi_I(\rho) \cdot \kappa_I(x)\ ,
	\end{align*}
	where we used the action of $\widetilde{W}_I \subset \Waff$ given in \eqref{eqn:affine_action} and the cocycle $c_I \colon W_I \to \Lambda$ defining the isomorphism $\varphi_I \colon W_I \to \widetilde{W}_I$. 
	
	Let $J \subseteq I \subseteq \{0,\dots, \ell\}$. Note that $G_I \subseteq G_J$, $W_I \subseteq W_J$ and $\hat{A}_I \subseteq \hat{A}_J$, where we use the notation from \eqref{eqn:AI_VI}. As explained in Sec.~\ref{subsec:Bredon} we also have $\widetilde{W}_I \subseteq \widetilde{W}_J$. To see that the diagram containing the differentials commutes it suffices to see the commutativity of 
	\[
		\begin{tikzcd}
			K_0^{G_J}(C^*\cE_{w_J}) \ar[r] \ar[d, "\kappa_J" left, "\cong" right] & K_0^{G_I}(C^*\cE_{w_I}) \ar[d, "\kappa_I" right, "\cong" left] \\
			R_F(\bT)^{\widetilde{W}_J} \ar[r] & R_F(\bT)^{\widetilde{W}_I}
		\end{tikzcd}
	\]
	where the lower horizontal arrow is given by the inclusion of fixed-points. The upper horizontal arrow is the following homomorphism: Consider the composition
	\[
		\begin{tikzcd}[column sep=0.9cm]
			K_0^{G_J}(C^*\cE_{\hat{A}_J}) \ar[r,"r^G_{IJ}" below] & K_0^{G_I}(C^*\cE_{\hat{A}_J}) \ar[r,"r^{\hat{A}}_{IJ}" below] & K_0^{G_I}(C^*\cE_{\hat{A}_I}) \ ,
		\end{tikzcd}
	\]  
	where $r^G_{IJ}$ and $r^{\hat{A}}_{IJ}$ are the maps obtained by restricting the group action and the base space of the Fell bundle, respectively. The upper arrow in the above diagram is then the composition of this homomorphism with the isomorphisms $K_0^{G_S}(C^*\cE_{\hat{A}_S}) \to K_0^{G_S}(C^*\cE_{z_S})$ for $S=I$ and $S=J$. Let $q_\bT \colon \bT \to \Delta^\ell$ send a point in $\bT$ to its conjugacy class and define $\widetilde{B}_J = q_\bT^{-1}(A_J)$ with $A_J$ as in \eqref{eqn:AI_VI}. Let $B_J = \exp^{-1}(\widetilde{B}_J)$, where $\exp \colon \liet \to \bT$ is the exponential map. The following diagram commutes
	\[
		\begin{tikzcd}[column sep=0.9cm,row sep=0.7cm]
			K_0^{G_J}(C^*\cE_{\hat{A}_J}) \ar[r] \ar[d,"\cong"] & K_0^{G_I}(C^*\cE_{\hat{A}_J}) \ar[r] \ar[d,"\cong"] & K_0^{G_I}(C^*\cE_{\hat{A}_I}) \ar[d,"\cong"] \\
			K_0^{\bT}(C^*\cE_{\widetilde{B}_J})^{W_J} \ar[r] \ar[d,"\cong"] & K_0^{\bT}(C^*\cE_{\widetilde{B}_J})^{W_I} \ar[r] \ar[d,"\cong"] & K_0^{\bT}(C^*\cE_{\widetilde{B}_I})^{W_I} \ar[d,"\cong"] \\
			K_0^{\bT}(C^*\cL_{B_J})^{W_J} \ar[r] & K_0^{\bT}(C^*\cL_{B_J})^{W_I} \ar[r] & K_0^{\bT}(C^*\cL_{B_I})^{W_I} 
		\end{tikzcd}
	\]  
	Note that $\exp^{-1}(w_J) = \xi_J + \Lambda$ and $C^*\cL_{w_J}$ is the $C^*$-algebra associated to the restriction of $\cL$ to $(\xi_J + \Lambda)^2 \subset \liet^{[2]}$. The $\Lambda$-equivariant bijection 
	\[
		\xi_J + \Lambda \to \xi_I + \Lambda
	\]
	that sends $\xi_J$ to $\xi_I$ lifts to a $W_I$-equivariant isomorphism $\cL_{w_J} \to \cL_{w_I}$ of Fell bundles giving a $*$-isomorphism $\psi_{IJ} \colon C^*\cL_{w_J} \to C^*\cL_{w_I}$. The group homomorphism $K_0^{\bT}(C^*\cL_{w_J})^{W_J} \to K_0^{\bT}(C^*\cL_{w_I})^{W_I}$ induced by the bottom row of the above diagram is the same as the one given by restricting the group action from $W_J$ to $W_I$ and then applying $\psi_{IJ}$. Consider the following diagram
	\[
		\begin{tikzcd}[row sep=0.5cm]
			K_0^{\bT}(C^*\cL_{w_J})^{W_J} \ar[r,"r^W_{IJ}"] \ar[d] & K_0^{\bT}(C^*\cL_{w_J})^{W_I} \ar[r,"\psi_{IJ}"] \ar[d] & K_0^{\bT}(C^*\cL_{w_I})^{W_I} \ar[dl] \\
			R_F(\bT)^{\widetilde{W}_J} \ar[r] & R_F(\bT)^{\widetilde{W}_I}
		\end{tikzcd}
	\]
	where the vertical and diagonal arrow are induced by the trivialisation of~$\cL_{w_J}$, respectively $\cL_{w_I}$. Since the trivialisation of $\cL_{w_J}$ is $W_J$-equivariant, the square in the diagram commutes. The homeomorphism $\xi_J + \Lambda \to \xi_I + \Lambda$ intertwines the two sections $\sigma_I \colon \xi_I + \Lambda \to \liet^{[2]}$ and $\sigma_J \colon \xi_J + \Lambda \to \liet^{[2]}$. This shows that the triangle in the diagram also commutes. Combining this diagram with the one from above proves the statement about the differentials.
\end{proof}

\subsection{Rationalisation and regular sequences}
Lemma~\ref{lem:cochain_complex} reduces the computation of the $E_1$-page of the spectral sequence to the computation of the cohomology of the cochain complex $C^\ast_{\Waff}(\liet; \cR)$. In fact, we will see in Thm.~\ref{thm:higher_twists} that rationally the spectral sequence collapses on the $E_2$-page. Since $\Waff = \Lambda \rtimes W$ and $W$ is finite, the cohomology computation can be dealt with in a two-step process after rationalisation. We will follow the argument given in \cite[Sec.~3]{paper:AdemCantareroGomez}. Let 
\begin{align*}
	\RQ &= R(\bT) \otimes \Q = \Q[t_1, \dots, t_n]/(t_1 \cdots t_n - 1) \ ,\\
	\RFQ &= R_F(\bT) \otimes \Q \ , \\
	\cRQ &= \cR \otimes \Q\ .
\end{align*}
In this section we will compute $H^*_{\Waff}(\liet, \cRQ)$ using regular sequences and relate it back to $K_*^G(C^*\cE) \otimes \Q$ later. Let $F \colon (\Viso, \oplus) \to (\Vgr, \otimes)$ be a non-trivial exponential functor, i.e.\ we have $\deg(F(t)) > 0$.

\begin{lemma} \label{lem:regseq1}
	Let $m \in \N$. The sequence 
	\[
		(t_2^m - t_1^m,\ t_3^m - t_2^m,\ \dots,\ t_{n-1}^m - t_{n-2}^m,\ - t_1^m\cdots t_{n-2}^mt_{n-1}^{2m})
	\] 
	is a regular sequence in $\hat{R}_\Q = \Q[t_1, \dots, t_{n-1}]$.
\end{lemma}

\begin{proof}
	Multiplication by $t_k^m - t_{k-1}^m$ is the same as multiplication by $t_k^m - t_{1}^m$ in the quotient  $\hat{R}_\Q/(t_2^m -t_1^m, \dots, t_{k-1}^m - t_{k-2}^m)$ for $k \in \{2,\dots, n-1\}$. Hence, it suffices to show that 
	\[
		(t_2^m - t_1^m,\ t_3^m - t_1^m,\ \dots,\ t_{n-1}^m - t_1^m,\ -t_1^{m}\cdots t_{n-2}^mt_{n-1}^{2m})
	\]
	is a regular sequence in $\hat{R}_\Q$. For $3 \leq k \leq n-1$ the quotient 
	\[
		\Q[t_1, \dots, t_{n-1}]/(t_2^m - t_1^m, \dots, t_{k-1}^m - t_1^m)
	\] 
	is free as a $\Q[t_1, t_k, \dots, t_{n-1}]$-module with basis 
	\[
		\{ t_2^{s_2} \cdots t_{k-1}^{s_{k-1}} \ | \ 0 \leq s_j \leq m-1 \text{ for all } j \in \{2, \dots, k-1\} \}\ .
	\]
	The multiplication by $t_k^m - t_1^m$ acts diagonally in the sense that it maps each basis element to a non-zero multiple. In particular, this map is injective.  
	
	The argument for $k = n$ is very similar. Here, the quotient
	\[
		\Q[t_1, \dots, t_{n-1}]/(t_2^m - t_1^m, \dots, t_{n-1}^m - t_1^m)
	\] 
	is again free as a $\Q[t_1]$-module with basis 
	\[
		\{ t_2^{s_2} \cdots t_{n-1}^{s_{n-1}} \ | \ 0 \leq s_j \leq m-1 \text{ for all } j \in \{2, \dots, n-1\} \}\ .
	\]
	The multiplication by $- t_1^m\cdots t_{n-2}^mt_{n-1}^{2m}$ is the same as multiplication by $-t_1^{nm}$, because $t_i^m = t_1^m$ in the quotient for $1 \leq i \leq n-1$. Following the same reasoning as above, this map is again injective. 
\end{proof}

Exactness of the localisation functor immediately yields the following corollary of Lem.~\ref{lem:regseq1}:

\begin{corollary} \label{lem:regseq2}
	The sequence 
	\[
		(F(t_2) - F(t_1), F(t_3) - F(t_2), \dots, F(t_n) - F(t_{n-1}))
	\]
	is regular in $\RFQ$. 
\end{corollary}

\begin{proof}
	Let $F(t) = \sum_{k=0}^m a_k t^k$ with $a_k \in \Q$, $a_m \neq 0$ and $\deg(F(t)) = m$. Note that the ring $\RQ$ is the localisation of $\hat{R}_{\Q}$ at $t_1 \cdots t_{n-1}$, i.e.\ in $\RQ$ we have $t_1\cdots t_n = 1$. Therefore $t_n^{-m}\,F(t_n) = (t_1\cdots t_{n-1})^m\,F(t_n) \in \hat{R}_\Q$. Moreover, $\RFQ$ is a localisation of $\RQ$. Since localization is an exact functor and $t_1, \dots, t_{n-1}$ are units in $\RFQ$, it suffices to show that 
	\[
		(F(t_2) - F(t_1), \dots, F(t_{n-1}) - F(t_{n-2}), (t_1\cdots t_{n-1})^m(F(t_{n}) - F(t_{n-1})))		
	\]
	is a regular sequence in $\hat{R}_\Q$. The commutative ring $\hat{R}_\Q$ has an $\N_0$-grading by the total degree. Thus, by \cite[Cor.~5.3]{paper:Hemmert} the regularity will follow if it holds for the sequence of homogeneous highest order terms:
	\[
		(a_m(t_2^m - t_1^m), \dots, a_m(t_{n-1}^m - t_{n-2}^m), - a_m t_1^m\cdots t_{n-2}^m\,t_{n-1}^{2m})\ .
	\]
	Since $a_m \neq 0$, the sequence agrees up to multiplication by a unit with the one from Lem.~\ref{lem:regseq1}. Thus, the regularity follows. 
\end{proof}

The following lemma reduces the computations of the $\Waff$-equivariant cohomology groups of $\liet$ with coefficient systems $\cRQ$ to computing the $W$-fixed points of $\Lambda$-equivariant cohomology. The proof makes use of our choice of rational coefficients. We claim no originality for this proof. It is a straightforward adaptation of \cite[Thm.~3.9]{paper:AdemCantareroGomez}. Nevertheless, we include the proof for the convenience of the reader. 

\begin{lemma}
Let $\psi \colon \Lambda \to GL_1(R_F(\bT))$ be the group homomorphism in \eqref{eqn:hom_to_GL1} and let $\cRQ$ be the coefficient system defined in \eqref{eqn:coeff_sys}. There is an isomorphism of $R_F(G) \otimes \Q$-modules
\begin{equation} \label{eqn:Waff_cohomology}
	H^*_{\Waff}(\liet; \cRQ) \cong H^*_{\Lambda}(\liet; \cRQ)^W\ .
\end{equation}
\end{lemma}

\begin{proof}
	Choose a CW-complex structure on $\liet$ such that the action of $\Waff$ is cellular. Let $C_*(\liet)$ be the associated cellular chain complex and denote by $C^*_{\Waff}(\liet; \cRQ)$ the complex of equivariant cochains, see \cite[Sec.~I.6]{book:BredonCohomology}. Let
	\[
		D^* = \text{Hom}(C_*(\liet), R_F(\bT) \otimes \Q)\ .
	\]
	The cochain complex $D^*$ has a $\Z$-linear action by $\Waff$ defined for $d \in D^k$, $w \in \Waff$ and $x \in C_k(\liet)$ by $(w \cdot d)(x) = w \cdot d(w^{-1}x)$, where the dot denotes the $\Waff$-module structure on $R_F(\bT) \otimes \Q$ given by \eqref{eqn:affine_action}. 
	By \cite[Sec.~I.9, eq.~(9.3)]{book:BredonCohomology} the cochain complex $C^*_{\Waff}(\liet; \cRQ)$ is isomorphic to 
	\[
		\text{Hom}_{\Z[\Waff]}(C_*(\liet), R_F(\bT) \otimes \Q) = (D^*)^{\Waff}\ .
	\]
	Since $\Lambda \subseteq \Waff$ is a normal subgroup with quotient $W$, we may take the fixed points in the final equation in two steps. We define
	\[
		E^* = \text{Hom}_{\Z[\Lambda]}(C_*(\liet), R_F(\bT) \otimes \Q)	= (D^*)^\Lambda	\ .
	\]
	Now consider the cohomology of $W \cong S_n$ with coefficients in the cochain complex $E^*$, i.e.\ $H^*(W; E^*)$. The double complex $C^*(G, E^*)$ computing it leads to two spectral sequences: the first one has $E_2$-term
	\[
		\,^IE_2^{p,q} = H^p(W; H^q(E^*))\ .
	\]
	Since $E^*$ is a cochain complex over $\Q$ and $W$ is finite, the groups on the $E_2$-page compute to 
	\[
		\,^IE_2^{p,q} \cong \begin{cases}
			H^q(E^*)^W & \text{if } p = 0 \ ,\\
			0 & \text{else} \ .
		\end{cases}
	\] 
	Interchanging horizontal and vertical directions leads to a second spectral sequence with $E_1$-page
	\[
		\,^{II}E_1^{p,q} = H^q(W; E^p) \cong \begin{cases}
			(E^p)^W & \text{if } q = 0\ , \\
			0 & \text{else}
		\end{cases}
	\]
	with differentials induced by the differential of the cochain complex $E^*$. Hence,
	\[
		\,^{II}E_2^{p,q} \cong \begin{cases}
			H^p((E^*)^W) & \text{if } q = 0 \ ,\\
			0 & \text{else}\ .
		\end{cases}
	\]
	Both spectral sequences collapse on the $E_2$-page without extension problems and both converge to $H^*(W; E^*)$. Therefore we have isomorphisms of $R(G) \otimes \Q$-modules
	\[
		H^*_{\Waff}(\liet; \cRQ) \cong H^*((E^*)^W) \cong H^*(E^*)^W \cong H^*_{\Lambda}(\liet; \cRQ)^W\ . \qedhere 
	\]
\end{proof}

Computing the $E_1$-page of the original spectral sequence \eqref{eqn:spec_seq} therefore reduces to determining the groups $H^*_{\Lambda}(\liet; \cRQ)$ and identifying the $W$-action on them. To determine $H^*_{\Lambda}(\liet; \cRQ)$ we will consider 
\[
	\liet = \{ (x_1, \dots, x_n) \in \R^n \ |\ x_1 + \dots + x_n = 0 \}
\]
as a $\Lambda$-CW-complex in the following way. Let $c_i = (0,\dots, 1, -1, \dots, 0) \in \Lambda$ for $i \in \{1, \dots, \ell\}$ be the element with a $1$ in the $i$th position. The $0$-cells are given by the lattice $\Lambda \subset \liet$. The $k$-cells for $k \geq 1$ are the elements in the $\Lambda$-orbit of the cube spanned by $c_{i_1}, \dots, c_{i_k}$ for each sequence $i_1 < i_2 < \dots < i_k$ with $i_\ell \in \{1, \dots, \ell\}$. We will denote the corresponding $k$-cell by $c_{i_1, \dots, i_k}$. Altogether we can identify integral $k$-chains of $\liet$ with
\[
	C_k(\liet) = \extp^k \Z^\ell \otimes \Z[\Lambda]\ ,
\]
where the exterior power is spanned by the cells $c_{i_1, \dots, i_k}$ and the second tensor factor keeps track of the position in the $\Lambda$-orbit. With $\Z[\Lambda] \cong \Z[s_1^{\pm 1}, \dots, s_\ell^{\pm 1}]$ the boundary operators of this chain complex turn out to be the ones of the Koszul complex for the sequence $(s_1 - 1, \dots, s_\ell - 1)$, i.e.\ 
\[
	\partial_k(c_{i_1, \dots, i_k}) = \sum_{j=1}^k (-1)^{j-1} (s_j-1)\, c_{i_1, \dots, \check{i}_j, \dots, i_k}\ .
\]
For $n =2$ the $\Lambda$-CW-structure and one boundary operator is illustrated in Fig.~\ref{fig:Koszul}. As discussed in the final remark of~\cite[I.9]{book:BredonCohomology} the $\Lambda$-equivariant cohomology of~$\liet$ is the cohomology of the cochain complex
\begin{equation} \label{eqn:hom_cochains}
	\hom_{\Q[\Lambda]}(C_*(\liet) \otimes \Q, \RFQ) \cong \extp^* \Q^\ell \otimes \RFQ \cong \extp^* (\RFQ)^{\ell} \ .
\end{equation}
By the definition of the $\Lambda$-action on $\RFQ$ the element $s_i \in \Z[\Lambda]$ is mapped to $F(t_i)F(t_{i+1})^{-1} \in \RFQ$. Therefore the coboundary operator in the above cochain complex is the one of the dual Koszul complex for the sequence $x_F = (F(t_1)F(t_2)^{-1}-1,\dots, F(t_{n-1})F(t_n)^{-1}-1) \in (\RFQ)^\ell$ given by 
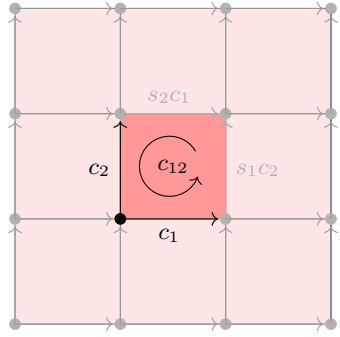
\[
	d_k(y) =  x_F \wedge y\ .
\] 

\begin{figure}[htp]
\centering
\[
\begin{tikzpicture} 	
	\draw[step=1.4cm,draw=gray,fill=red!10,very thin] (-1.4,-1.4) grid (2.8,2.8) rectangle (-1.4,-1.4);
	\path[fill=red!40] (0,0) rectangle (1.4,1.4);
	\foreach \x in {-1,...,2}
    	\foreach \y in {-1,...,2}  
			\draw[color=black!30,fill=black!30] (1.4*\x,1.4*\y) circle (2pt);
	\foreach \x in {-1,...,2}
    	\foreach \y in {-1,...,1}  
			\draw[black!40,->] (1.4*\x,1.4*\y) -- (1.4*\x,1.4*\y + 1.3);
	\foreach \x in {-1,...,1}
    	\foreach \y in {-1,...,2}  
			\draw[black!40,->] (1.4*\x,1.4*\y) -- (1.4*\x + 1.3,1.4*\y);

	\draw[color=black,fill=black] (0,0) circle (2pt);
	\draw[black,->] (0,0) -- (1.3,0) node[midway,below,font=\footnotesize] {$c_1$};
	\draw[black,->] (0,0) -- (0,1.3) node[midway,left,font=\footnotesize] {$c_2$};
	\draw[black!30,->] (1.4,0) -- (1.4,1.3) node[midway,right,font=\footnotesize] {$s_1 c_2$};
	\draw[black!30,->] (0,1.4) -- (1.3,1.4) node[midway,above,font=\footnotesize] {$s_2 c_1$};
	\node[black,font=\footnotesize] at (0.7,0.7) {$c_{12}$};
	
	\draw[->] (1,0.9) arc (30:340:0.4cm);
	
	\node[right] at (4,1) {$\partial c_{12} = c_1 + s_1c_2 - s_2 c_1 - c_2$};
	\node[right] at (4.85,0.3) {$= (s_1-1)c_2 - (s_2-1)c_1$};
\end{tikzpicture}
\]
\caption{\label{fig:Koszul} The $\Lambda$-CW-complex structure of $\liet$ for $n=2$ and the boundary of the $2$-cell $c_{12}$.}
\end{figure}

\begin{lemma} \label{lem:coh_computation}
	We have $H^k_{\Lambda}(\liet; \cRQ) = 0$ for $k \neq \ell$ and 
	\[
		H^{\ell}_{\Lambda}(\liet; \cRQ) \cong \RFQ / (F(t_2) - F(t_1), \dots, F(t_n) - F(t_{n-1}))
	\]
	Moreover, the Weyl group $W \cong S_n$ acts on $H^{\ell}_{\Lambda}(\liet; \cRQ)$ by signed permutations of the variables $t_1, \dots, t_n$.
\end{lemma}

\begin{proof}
The elements $F(t_i) \in \RFQ$ are invertible. Therefore the sequence 
\[
	(F(t_1)F(t_2)^{-1} - 1, \dots, F(t_{n-1})F(t_n)^{-1}-1)
\]
is regular by Lem.~\ref{lem:regseq2}. Hence, the first statement follows from classical results about regular sequences, see for example \cite[Cor.~17.5]{book:Eisenbud}. 

Let $I_{F,\Q} = (F(t_1)F(t_2)^{-1} - 1, \dots, F(t_{n-1})F(t_n)^{-1}-1)$ be the ideal generated by the sequence. The permutation action of $W$ on $\liet$ restricts to an action on $\extp^k\Z^\ell$ for each $k \in \{0,\dots, \ell\}$, which extends to an action on $C_*(\liet)$. Hence, $W$ also acts on the cochain complex $\hom_{\Z[\Lambda]}(C_{\ell}(\liet), \RFQ)$ by conjugation and the isomorphism 
\[
	H^\ell_\Lambda(\liet; \cRQ) \cong \RFQ / I_{F,\Q}
\]
is $W$-equivariant if the right hand side is equipped with the signed permutation action. 

In the quotient on the right hand side translations along the roots act trivially. This implies that the cellular action of $W$ on the subdivision of $C_{\ell}(\liet)$ that uses all positive roots induces the same action on the cohomology of the cochain complex $\hom_{\Z[\Lambda]}(C_{\ell}(\liet), \RFQ)$ as the one described above. 
\end{proof}

Let $I_{F,\Q} = (F(t_2) - F(t_1), \dots, F(t_n) - F(t_{n-1}))$ as in the last lemma. We adopt the following notation to distinguish the ordinary permutation action from its signed counterpart: If we consider $\RFQ$ with the action by signed permutations, then we denote it by $\sgn{\RFQ}$, otherwise by $\RFQ$. Note that the ideal $I_{F,\Q}$ is invariant under both $W$-actions. Hence, we will write $\sgn{I_{F,\Q}}$ if we consider it with the signed permutation action. Taking invariants with respect to the action of a finite group is an exact functor on rational vector spaces. Thus, Lem.~\ref{lem:coh_computation} immediately gives
\[
	H^{\ell}_{\Waff}(\liet; \cRQ) \cong H_{\pi_1(\bT)}^{\ell}(\liet; \cRQ)^W \cong (\sgn{\RFQ})^W / (\sgn{I_{F,\Q}})^W\ .
\]
Let $\Delta \in \Q[t_1, \dots, t_n]$ be the Vandermonde determinant. Multiplication by $\Delta^{-1}$ induces an isomorphism of $\RFQ^W$-modules
\begin{equation} \label{eqn:antisym_to_sym}
	\Psi \colon (\sgn{\RFQ})^W \to \RFQ^W \quad, \quad p \mapsto \frac{p}{\Delta}\ .
\end{equation}

Let $p \in \Q[t]$. As we will see in the next lemma we will need extended versions of the Schur polynomials, which are defined as follows: 
\begin{equation} \label{eqn:def_a}
	a_{(p, \lambda_2 \dots, \lambda_n)}(t_1, \dots, t_n) = 
	\det\begin{pmatrix}
		p(t_1) & p(t_2) & \dots & p(t_n) \\
		t_1^{\lambda_2 + n-2} & t_2^{\lambda_2 + n-2} & \dots & t_n^{\lambda_2 + n-2} \\
		t_1^{\lambda_3 + n-3} & t_2^{\lambda_3 + n-2} & \dots & t_n^{\lambda_3 + n-3} \\
		\vdots & \vdots & \ddots & \vdots \\
		t_1^{\lambda_n} & t_2^{\lambda_n} & \dots & t_n^{\lambda_n} 
	\end{pmatrix}\ .
\end{equation}

\begin{lemma} \label{lem:PieriRule}
	Let $p \in \Q[t]$, let $q(t) = p(t)t$ and let $e_1(t_1, \dots, t_n) = t_1 + \dots + t_n$ be the first elementary symmetric polynomial. Then
	\[
		a_{(p, 1,0, \dots, 0)} = a_{(p,0,0, \dots, 0)} \cdot e_1 - a_{(q,0,0, \dots, 0)}\ . 
	\]	 
\end{lemma}

\begin{proof}
	Both sides are $\Q$-linear in $p$. Hence, it suffices to consider the monomials $p(t) = t^k$ for $k \in \N_0$. For $k \in \{0, \dots, n-3\}$ both sides vanish. 
	
	For $p(t) = t^{n-2}$ the term $a_{(p,0,0,\dots, 0)}$ vanishes, $q(t) = t^{n-1}$ and the matrix underlying $a_{(p,1,0,\dots,0)}$ is obtained from the one for $a_{(q,0,0, \dots, 0)}$ by interchanging the first two rows, producing a sign in the determinant. Hence, the equation holds in this case as well.  

	Now let $p(t) = t^k$ with $k \geq n-1$. In this case, we can express $a$ in terms of Schur polynomials:
	\[
		\frac{1}{\Delta}\,a_{(p, \lambda_2, \dots, \lambda_n)} = s_{(k-(n-1), \lambda_2, \dots, \lambda_n)}\ .
	\]
	By Pieri's rule we have
	\[
		s_{(k-(n-1), 0,0, \dots, 0)} \cdot e_1 = s_{(k+1-(n-1), 0,0, \dots, 0)} + s_{(k-(n-1), 1,0, \dots, 0)}
	\]
	and therefore also \( a_{(p, 1,0, \dots, 0)} = a_{(p,0,0, \dots, 0)} \cdot e_1 - a_{(q,0,0, \dots, 0)} \).
\end{proof}

\begin{lemma} \label{lem:generators}
	The $\RFQ^W$-submodule $(\sgn{I_{F,\Q}})^W$ of $(\sgn{\RFQ})^W$ is generated by the $n-1$ antisymmetric polynomials $q_i$ for $i \in \{0,\dots, n-2\}$ defined by
	\[
		q_i(t_1, \dots, t_n) = \det 
		\begin{pmatrix}
			F(t_1)t_1^i & F(t_2)t_2^i & \dots & F(t_n)t_n^i \\
			t_1^{n-2} & t_2^{n-2} & \dots & t_n^{n-2} \\
			\vdots & \vdots & \ddots & \vdots \\
			t_1 & t_2 & \dots & t_n \\
			1 & 1 & \dots & 1
		\end{pmatrix}\ .
	\]
\end{lemma}

\begin{proof}
	Let $R_\Q = R(\bT) \otimes \Q \cong \Q[t_1, \dots, t_n]/(t_1\cdots t_n - 1)$ and let $I_\Q \subseteq R_\Q$ be the ideal generated by $(F(t_2) - F(t_1), \dots, F(t_n) - F(t_{n-1}))$. As above $S_n$ acts by permutations or by signed permutations on $R_\Q$ and $I_\Q$. If $S_n$ acts on $R_\Q$ and~$I_\Q$ via signed permutations, we denote this by $\sgn{R_\Q}$ and $\sgn{I_\Q}$, respectively. 
	
	Observe that $F(t_1 + \dots + t_n) \in \Q[t_1, \dots, t_n]^W$ and let 
	\[
		((\sgn{R_\Q})^W)_F = (\sgn{R_\Q})^W[F(t_1 + \dots + t_n)^{-1}]\ .
	\] 
	We claim that $((\sgn{R_\Q})^W)_F \cong (\sgn{\RFQ})^W$ and will prove this first:
	The $W$-equivariant $R_\Q^W$-module homomorphism $\sgn{R_\Q} \to \sgn{\RFQ}$ induces a homomorphism $(\sgn{R_\Q})^W \to (\sgn{\RFQ})^W$. Since multiplication by $F(t_1 + \dots + t_n)$ is invertible in the codomain, this map gives rise to the module homomorphism $((\sgn{R_\Q})^W)_F \to (\sgn{\RFQ})^W$, which is injective, because it can be obtained from a restriction of the injective map $\sgn{R_\Q} \to \sgn{\RFQ}$ by localisation, which is exact. Let 
	\[
		\frac{p}{q} \in (\sgn{\RFQ})^W
	\]
	with $p \in \sgn{R_\Q}$ and $q = F(t_1 + \dots + t_n)^{k}$ for some $k \in \N_0$. Since $q$ is $W$-invariant, the condition $\sigma \cdot \tfrac{p}{q} = \tfrac{p}{q}$ implies $\sigma \cdot p = p$ for all $\sigma \in W$. Thus $ \tfrac{p}{q} \in ((\sgn{R_\Q})^W)_F$. Hence, $((\sgn{R_\Q})^W)_F \to (\sgn{\RFQ})^W$ is an isomorphism. Therefore it suffices to show that $(\sgn{I_\Q})^W \subset (\sgn{R_\Q})^W$ is the $R_\Q^W$-submodule generated by $q_0, \dots, q_{n-2} \in (\sgn{R_\Q})^W$.  
	
	Let $P_\Q = \Q[t_1, \dots, t_n]$ and denote by $\sgn{P}_\Q$ the $P_\Q$-module equipped with its natural $W$-action by signed permutations. Observe that the quotient map $\pi \colon \sgn{P_\Q} \to \sgn{R_\Q}$ is $W$-equivariant. Let $J_\Q \subset P_\Q$ be the ideal generated by $(F(t_2) - F(t_1), \dots, F(t_n) - F(t_{n-1}))$. Note that $\pi(J_\Q) = I_\Q$. Hence, it suffices to see that $(\sgn{J_\Q})^W \subset (\sgn{P_\Q})^W$ is the $P_\Q^W$-submodule generated by $q_0, \dots, q_{n-2} \in (\sgn{P_\Q})^W$.  
	
	Now consider the antisymmetrisation map, i.e.\ the $P_\Q^W$-module homomorphism that averages over the $W$-action: 
	\[
		\theta \colon P_\Q \to (\sgn{P_\Q})^W \quad, \quad p \mapsto \frac{1}{n!} \sum_{\sigma \in W} \sigma \cdot p\ .
	\] 
	It is surjective and maps $J_\Q$ onto $(\sgn{J_\Q})^W$. Hence, it suffices to prove that $(\sgn{J_\Q})^W = (q_0, \dots, q_{n-2})$. We will first show that $(\sgn{J_\Q})^W \subseteq (q_0, \dots, q_{n-2})$. Note that $\theta$ is anti-equivariant in the sense that $\theta(\sigma \ast p) = \text{sign}(\sigma)\,\theta(p)$, where $\sigma \ast p$ denotes the (unsigned) permutation action of $\sigma \in W$ on $p \in P_\Q$. It suffices to show that for all $p_2, \dots, p_n \in P_\Q$
	\[
		\theta( (F(t_2) - F(t_1))\,p_2 + \dots + (F(t_n) - F(t_{n-1}))\,p_n ) \in (q_0, \dots, q_{n-2})\ .
	\] 
	For each $i \in \{2, \dots, n\}$ let $\sigma_i \in W$ be the permutation interchanging $2 \leftrightarrow i$ and $1 \leftrightarrow (i-1)$. Then 
	\[
		\theta( (F(t_i) - F(t_{i-1}))p_i ) = \pm \theta( (F(t_2) - F(t_{1})) (\sigma_i \ast p_i) )\ .
	\] 
	Hence, by linearity of $\theta$ it suffices to consider $\theta( (F(t_2) - F(t_1))\,p)$ with $p \in P_\Q$. 
	
	Note that $P_\Q$ is a free $P_\Q^W$-module with basis $\{ t_1^{k_1}\cdots t_n^{k_n}\ |\ 0 \leq k_i \leq n-i \}$ by \cite[p.~41]{book:ArtinGalois}. Thus, it suffices to see that each $\theta( (F(t_2) - F(t_1))\,t_1^{k_1}\cdots t_n^{k_n})$ is a $P_\Q^W$-linear combination of $q_0, \dots, q_{n-2}$. 
	
	We start by computing $\theta(F(t_1)\,t_1^{k_1} \cdots t_n^{k_n})$. By the Leibniz formula this antisymmetrisation can be written as the following determinant:
	\begin{equation}
		\label{eqn:det1}
		\det\begin{pmatrix}
			F(t_1)t_1^{k_1} & F(t_2)t_2^{k_1} & \dots & F(t_n)t_n^{k_1} \\
			t_1^{k_2} & t_2^{k_2} & \dots & t_n^{k_2} \\
			\vdots & \vdots & \ddots & \vdots \\
			t_1^{k_n} & t_2^{k_n} & \dots & t_n^{k_n} 
		\end{pmatrix}\ .
	\end{equation}
	For this determinant to be non-zero the condition $0 \leq k_j \leq n-j$ enforces $k_j = n-j$ for $j \in \{2,\dots,n\}$, otherwise we would have two identical rows in the matrix. In the same way, $\theta(t_1^{k_1}F(t_2)t_2^{k_2}\,t_3^{k_3} \cdots t_n^{k_n})$ can be expressed as the determinant
	\begin{equation}	
		\label{eqn:det2}
		\det\begin{pmatrix}
			t_1^{k_1} & t_2^{k_1} & \dots & t_n^{k_1} \\
			F(t_1)t_1^{k_2} & F(t_2)t_2^{k_2} & \dots & F(t_n)t_n^{k_2} \\
			t_1^{k_3} & t_2^{k_3} & \dots & t_n^{k_3} \\
			\vdots & \vdots & \ddots & \vdots \\
			t_1^{k_n} & t_2^{k_n} & \dots & t_n^{k_n} 
		\end{pmatrix}\ .
	\end{equation}
	Thus, $\theta( (F(t_2) - F(t_1))\,t_1^{k_1}\cdots t_n^{k_n}) = 0$ unless $k_j = n-j$ for $j \in \{3, \dots, n\}$. If $k_1 \in \{0, \dots, n-3\}$, then 
	\begin{equation} \label{eqn:qk1}
		\theta( (F(t_2) - F(t_1))\,t_1^{k_1}\cdots t_n^{k_n}) = -q_{k_1}(t_1, \dots, t_n)\ ,
	\end{equation}
	because the antisymmetrisation of the term $F(t_2)t_1^{k_1}\cdots t_n^{k_n}$ vanishes. 
	
	If $k_1 = n-2$, then 
	\begin{equation} \label{eqn:qkn-2}
		\theta( (F(t_2) - F(t_1))\,t_1^{k_1}\cdots t_n^{k_n}) = \begin{cases}
			-2q_{n-2}(t_1, \dots, t_n) & \text{if } k_2 = n-2 \ ,\\
			-q_{k_2}(t_1, \dots, t_n) & \text{if } k_2 \neq n-2\ .
		\end{cases}
	\end{equation}
	The only remaining case is therefore $k_1 = n-1$ and $k_2 \in \{0, \dots, n-2\}$. For $k_2 \in \{0,\dots, n-3\}$ the determinant in \eqref{eqn:det1} vanishes and interchanging the first two rows in \eqref{eqn:det2} we see that the rest is equal to $-a_{(p,1,0, \dots, 0)}$ for $p(t) = F(t)t^{k_2}$ as defined in \eqref{eqn:def_a}. By Lem.~\ref{lem:PieriRule} we have for $q(t) = F(t)t^{k_2+1}$
	\begin{align*}
		&\theta( (F(t_2) - F(t_1))\,t_1^{k_1}\cdots t_n^{k_n}) = -a_{(p,1,0, \dots, 0)}(t_1, \dots, t_n) \\
		=\ & (a_{(p,0,\dots, 0)}\cdot e_1)(t_1, \dots, t_n) - a_{(q,0,\dots, 0)}(t_1, \dots, t_n) \\
		=\ & q_{k_2}(t_1, \dots, t_n) \cdot e_1(t_1,\dots, t_n) - q_{k_2+1}(t_1,\dots, t_n)\ .
	\end{align*}
	Finally, for $k_1 = n-1$, $k_2 = n-2$, $p(t) = F(t)t^{n-2}$ and $q(t) = F(t)^{n-1}$ we have 
	\begin{align*}
		\theta( (F(t_2) - F(t_1))\,t_1^{k_1}\cdots t_n^{k_n}) & = \left(-a_{(p,1,0,\dots,0)} + a_{(q,0,\dots, 0)}\right)(t_1,\dots, t_n) \\
		& = q_{n-2}(t_1, \dots, t_n) \cdot e_1(t_1,\dots, t_n)\ .
	\end{align*}
	
	Along the way we have also shown that $(q_0, \dots, q_{n-2}) \subseteq (\sgn{J_\Q})^W$, since we have written the polynomials $q_i$ for $i \in \{0,\dots, n-2\}$ as scalar multiples of  antisymmetrisations of the form $\theta( (F(t_2) - F(t_1))\,t_1^{k_1}\cdots t_n^{k_n})$, see \eqref{eqn:qk1} for $i \in \{0,\dots, n-3\}$ and \eqref{eqn:qkn-2} for $i = n-2$.
\end{proof}

We can now summarise the observations from this and the previous sections as follows:
\begin{theorem} \label{thm:higher_twists}
	The rational graded higher twisted $K$-theory of $G = SU(n)$ for a twist induced by an exponential functor $F \colon (\Viso,\oplus) \to (\Vgr, \otimes)$ with $\deg(F(t)) > 0$ is given by
	\begin{align*}
		K_{n-1}^G(C^*\cE) \otimes \Q &\cong R_F(G) \otimes \Q / J_{F,\Q} \ ,\\
		K_n^G(C^*\cE) \otimes \Q &\cong 0\ ,
	\end{align*}
	where we identify $J_{F,\Q}$ with the ideal in $R_F(G) \otimes \Q$ obtained as the image under the isomorphism $(\RFQ)^W \cong R_F(G) \otimes \Q$. This ideal is generated by 
	\[
		J_{F,\Q} = \left( \frac{q_0}{\Delta}, \dots, \frac{q_{n-2}}{\Delta} \right) \subset (\RFQ)^W
	\]	
	where $\Delta$ is the Vandermonde determinant and the polynomials $q_i$ are defined in Lem.~\ref{lem:generators}.
\end{theorem}

\begin{proof}
Continuity of $K$-theory implies $K_*^G(C^*\cE) \otimes \Q \cong K_*^G(C^*\cE \otimes \mathcal{Q})$ for the universal UHF-algebra $\mathcal{Q}$ (equipped with the trivial $G$-action). The $E_1$-page of the spectral sequence \eqref{eqn:spec_seq} for $C^*\cE \otimes \mathcal{Q}$ is given by
\[
	E_1^{p,q} = \begin{cases}
		\bigoplus_{\lvert I \rvert = p+1} R_F(G_I) \otimes \Q & \text{for } q \text{ even}, \\
		0 & \text{for } q \text{ odd}\ . 
	\end{cases}
\]
For even $q$ the lines in the spectral sequence boil down to the cochain complex $C^p_{\Waff}(\liet, \cRQ)$ by Lem.~\ref{lem:cochain_complex}, which computes the $\Waff$-equivariant Bredon cohomology of $\liet$. As observed in \eqref{eqn:Waff_cohomology} we have
\[
	H^p_{\Waff}(\liet, \cRQ) \cong H^p_\Lambda(\liet, \cRQ)^W
\]
which is only non-trivial for $p = n-1$ by Lem.~\ref{lem:coh_computation}. In particular, the spectral sequence collapses on the $E_2$-page giving $K_{n-1}^G(C^*\cE) \otimes \Q \cong H^{n-1}_\Lambda(\liet, \cRQ)^W$ and $K_n^G(C^*\cE) \otimes \Q = 0$. Combining Lem.~\ref{lem:coh_computation} and Lem.~\ref{lem:generators} we obtain 
\[
	H^{n-1}_{\Lambda}(\liet, \cRQ)^W \cong (\sgn{\RFQ})^W/(\sgn{I_{F,\Q}})^W
\]
with $(\sgn{I_{F,\Q}})^W = (q_0, \dots, q_{n-2})$. The $\RFQ^W$-module isomorphism $\Psi$ from~\eqref{eqn:antisym_to_sym} maps the submodule $(\sgn{I_{F,\Q}})^W$ to the ideal $J_{F,\Q}$ in $(\RFQ)^W$. This proves the statement.
\end{proof}

\begin{remark} \label{rem:gen_classical}
	Let us briefly discuss the case of classical twists $SU(n)_k$, i.e.\ $SU(n)$ at level $k$. Because the dual Coxeter number of $SU(n)$ is $n$, the exponential functor corresponding to the classical twist at level $k$ is
	\[
		F = \left(\extp^{\rm top}\right)^{\otimes (n+k)}\ .
	\] 
	With $F(t_j) = (-t_j)^{n+k}$ each generator gives (up to a sign) an alternating polynomial $a_{(\lambda_1, \dots, \lambda_n)}$ as follows:
	\begin{align*}
		& q_i(t_1, \dots, t_n) = (-1)^{n+k}\,a_{(k+i+1,0,\dots,0)}(t_1, \dots, t_n) \\
	 \Rightarrow \quad & \frac{q_i}{\Delta} = (-1)^{n+k}\,s_{(k+i+1,0,\dots,0)} = (-1)^{n+k}\,h_{k+i+1}
	\end{align*}
	Here, $s_{(\lambda_1,\dots,\lambda_n)}$ is the Schur polynomial and the final equation follows from the fact that $s_{(m,0,\dots,0)} = h_m$ for the complete homogeneous symmetric polynomial~$h_m$. Denote the $i$th fundamental weight of $SU(n)$ by $\omega_i$ with $\omega_i(x_1, \dots, x_n) = x_1 + \dots + x_i$ for $(x_1, \dots, x_n) \in \liet$. By the Weyl character formula $h_{k+i+1}$ is the character polynomial of the representation with highest weight $(k+i+1)\,\omega_1$. Hence, our generators agree (up to signs) with the ones found in \cite[Sec.~3.2]{paper:BouwknegtRidout}. 
\end{remark}

\section{Potentials and links to loop group representations}
It is surprising that the rational graded higher twisted $K$-groups of $SU(n)$ still carry a ring structure and are in fact quotients of a localisation of the representation ring. These properties are known to hold for classical twists and seem to be preserved when allowing higher ones. In this section we will see that it is still possible to find a potential generating the ideals underlying higher twists. In addition, we will construct a non-commutative counterpart of the determinant bundle over $LSU(n)$ from the full exterior algebra functor that generalises the central extension classified by the level.

\subsection{A potential for higher twists} \label{sec:potential}
As noted in Thm.~\ref{thm:higher_twists} the $\RFQ^W$-module isomorphism $\Psi$ from~\eqref{eqn:antisym_to_sym} maps the submodule $(\sgn{I_{F,\Q}})^W$ to the ideal $J_{F,\Q}$ in $\RFQ^W$ given by 
\[
	J_{F,\Q} = \left( \frac{q_0}{\Delta}, \dots, \frac{q_{n-2}}{\Delta} \right)\ .
\]
These generators can now be expressed in terms of symmetric polynomials as follows. Let 
\[
	F(t) = \sum_{i=0}^d \mu_i t^i\ .
\]
Let $c_{F,j} = \frac{q_j}{\Delta}$. With $a_{(p, \lambda_2, \dots, \lambda_n)}$ as in \eqref{eqn:def_a} we have 
\begin{align*}
	c_{F,j} & = \frac{1}{\Delta} a_{(F(t)t^j, 0, \dots, 0)} = \sum_{i=0}^d \mu_i\,\frac{1}{\Delta} a_{(t^{i+j}, 0, \dots, 0)} = \sum_{i=0}^d \mu_i s_{(i+j-(n-1), 0,\dots,0)} \\
	& = \sum_{i=0}^d \mu_i c_{i+j-(n-1)} \underset{(*)}{=} \sum_{i=1}^d \mu_i c_{i+j-(n-1)}\ ,
\end{align*}
where $c_k(t_1, \dots, t_n) = \sum_{1 \leq i_1 \leq \dots \leq i_k \leq n} t_{i_1} t_{i_2} \cdots t_{i_k}$ denotes the complete homogeneous symmetric polynomials and $s_{(\lambda_1, \dots, \lambda_n)}$ is the Schur polynomial. Note that $a_{(t^m,0, \dots, 0)} = 0$ for $m < n-1$. Hence, we define $c_k=0$ if $k < 0$. In particular, the summand $\mu_0 c_{j - (n-1)}$ vanishes for $j \in \{0,\dots, n-2\}$, which explains the equality $(*)$. Define
\[
	\bar{c}_k(t_1, \dots, t_n) = \sum_{1 \leq i_1 < \dots < i_k \leq n} t_{i_1} t_{i_2} \cdots t_{i_k}
\]
to be the elementary symmetric polynomials. The notation here is chosen to reflect the connection between the generators of $J_{F,\Q}$ and the universal Chern classes and follows Gepner \cite{paper:Gepner}. Note that $R_\Q^W \cong \Q[\bar{c}_1, \dots, \bar{c}_{n-1}]$ and therefore
\[
	\RFQ^W \cong \Q[\bar{c}_1, \dots, \bar{c}_{n-1}, F(\bar{c}_1)^{-1}]\ .
\]

\begin{prop} \label{prop:potential}
	Let $F \colon (\Viso,\oplus) \to (\Vgr,\otimes)$ be an exponential functor. Let $F(t) \in \Q[t]$ be the character polynomial of $F$ and let $G(t) \in \Q[t]$ be any polynomial integrating $\frac{F(t) - F(0)}{t}$, i.e.\ $G$ satisfies $G'(t) = \frac{F(t) - F(0)}{t}$. Define
	\[
		V(t_1,\dots, t_n) = \sum_{i=1}^n G(t_i) \in \RFQ^W\ .
	\]
	This potential generates the ideal $J_F$ in the sense that for $j \in \{0,\dots, n-2\}$
	\[
		c_{F,j} = (-1)^{n-j}\frac{\partial V}{\partial \bar{c}_{n-(j+1)}} 
	\]
\end{prop}

\begin{proof}
	Let $V_m(t) = \frac{1}{m}\,\sum_{k=1}^n t_k^m$. The computation in \cite[p.~389]{paper:Gepner} shows 
	\begin{equation} \label{eqn:Vm}
		\frac{\partial V_m}{\partial \bar{c}_j} = (-1)^{j-1} c_{m - j}\ .
	\end{equation}
	Let $F(t) = \sum_{i=0}^d \mu_i t^i$ and note that $\frac{F(t) - F(0)}{t} = \sum_{i=1}^d \mu_i t^{i-1}$. Since we may neglect constant terms, we can without loss of generality assume $G(0) = 0$. We have
	\[
		V(t_1, \dots, t_n) = \sum_{i=1}^d \mu_i \sum_{k=1}^n \frac{t_k^{i}}{i} = \sum_{i=1}^d \mu_i V_{i}\ .
	\]
	Using \eqref{eqn:Vm} the derivatives evaluate for $j \in \{0,\dots, n-2\}$ to
	\[
		\frac{\partial V}{\partial \bar{c}_{n-(j+1)}} = \sum_{i=1}^d \mu_i \frac{\partial V_{i}}{\partial \bar{c}_{n-(j+1)}} = (-1)^{n-j} \sum_{i=1}^d \mu_i c_{i+j-(n-1)} = (-1)^{n-j}c_{F,j}\ . \qedhere
	\]
\end{proof}

\begin{remark}
	The case $SU(n)_{k}$ corresponds to $F = \left(\extp^{\rm top}\right)^{\otimes (n+ k)}$ and therefore $F(t) = (-t)^{n + k}$, where $k \in \N_0$ is the level, $n$ is equal to the dual Coxeter number for $G = SU(n)$ and $t$ is considered to be odd as reflected by the sign. In this case the potential is
	\[
		V(t_1, \dots, t_n) = \frac{(-1)^{n+k}}{n+k} \sum_{i=1}^n t_i^{n+k} = V_{n + k}(t_1, \dots, t_n)\ .
	\]
	which coincides (up to the choice of sign) with the potential in \cite{paper:Gepner}. The generators $c_{F,j}$ boil down to 
	\(
		c_{F,j} = (-1)^{n+k}c_{k+j+1}
	\) for $j \in \{0,\dots, n-2\}$ and Prop.~\ref{prop:potential} retrieves the result from~\cite{paper:Gepner} (up to sign) that 
	\[
		\frac{\partial V}{\partial \bar{c}_{n-(j+1)}} = (-1)^{n-j} c_{F,j} = (-1)^{k-j} c_{k + j + 1} \ .
	\]
\end{remark}

\subsection{Exponential functors as higher determinants} \label{subsec:HigherDet}
This section will be more speculative than the previous ones. We will outline how exponential functors can be used to construct higher determinant bundles over loop groups. We will also briefly address the multiplicativity of our higher twists.

Let $H = H_+ \oplus H_-$ be a separable $\Z/2\Z$-graded Hilbert space with $\dim(H_{\pm}) = \infty$. The \emph{reduced general linear group} $\GLres{H} \subset GL(H)$ consists of operators
\[
	T = \begin{pmatrix}
		a & b \\
		c & d
	\end{pmatrix}
\]
such that $b$ and $c$ are Hilbert-Schmidt (which implies that $a$ and $b$ are Fredholm operators, see \cite[p.~81]{book:PressleySegal}). Consequently, the path-component of the identity $\GLidres{H}$ contains operators $T$ as above such that the index of $a$ vanishes. The unitary counterpart of $\GLres{H}$ is 
\[
	\Ures{H} = \GLres{H} \cap U(H)
\]	
with identity component $\Uidres{H}$ \cite[Def.~6.2.3]{book:PressleySegal}. The group $\Uidres{H}$ fits into a short exact sequence
\begin{equation} \label{eqn:Udet_ext}
	\Udet{H} \to \cE \to \Uidres{H}\ ,
\end{equation}
where
\begin{gather*}
	\Udet{H} = \{q \colon H_+ \to H_+\ :\ q \in U(H_+) \text{ and } q-1 \text{ is trace-class}\} \ ,\\
	\cE = \{ (T,q) \in \Uidres{H} \times U(H_+)\ :\ T = \left(\begin{smallmatrix}
		a & b \\ c & d
	\end{smallmatrix}\right) \text{ and } a - q \text{ is trace-class} \}\ .
\end{gather*}
The group $\Udet{H}$ consists of those unitary operators that have a Fredholm determinant, which provides a group homomorphism 
\(
	\det \colon \Udet{H} \to U(1)
\).
If $u = 1 + t$ for a trace-class operator $t$, then it is given by 
\[
	\det(u) = \sum_{n=0}^\infty \text{Tr}\left(\extp^n t\right)\ .
\]
The determinant can be used to obtain the determinant bundle over $\Uidres{H}$ from the extension \eqref{eqn:Udet_ext} given by 
\begin{equation} \label{eqn:U1-Ures0}
	U(1) \to \cE \times_{\det} U(1) \to \Uidres{H}\ .
\end{equation}
Since $\det$ is surjective, we have $\cE \times_{\det} U(1) \cong \cE/\ker(\det)$. Hence, \eqref{eqn:U1-Ures0} is in fact a central $U(1)$-extension of groups \cite[Sec.~6.6]{book:PressleySegal}.

With $G = SU(n)$ let $H = L^2(S^1, \C^n)$. The Fourier decomposition of vectors in $H$ gives rise to a $\Z/2\Z$-grading by letting $H_+$ be the subspace generated by $z^k$ with $k \geq 0$ and $z = e^{i\theta}$. The smooth loop group $LG$ acts by multiplication operators on $H$. Since $G = SU(n)$ is simply-connected, $LG$~is path-connected. In fact, it is not difficult to show (see \cite[Prop.~6.3.1]{book:PressleySegal}) that the representation by multiplication operators factors through
\[
	\rho \colon LG \to \Uidres{H}\ .
\]
The central $U(1)$-extension that gives rise to the level is the pullback of \eqref{eqn:U1-Ures0} with respect to $\rho$, see \cite[Sec.~6.6]{book:PressleySegal}.

We will outline how exponential functors give rise to $C^*$-algebraic counterparts of~\eqref{eqn:U1-Ures0}. For simplicity we will restrict ourselves to the exterior algebra functor $F = \extp^*$, even though the arguments given below will work for a much larger class. Note that $\extp^*$ extends to a functor on the category of countably infinite-dimensional Hilbert spaces and unitary isomorphisms. 

Now take $H = \C^n \otimes \ell^2(\N)$ and denote the canonical Hilbert basis of $\ell^2(\N)$ by $\{e_n \}_{n \in \N}$. Identify $V_k = \bigoplus_{k=1}^m\C^n$ with the subspace spanned by vectors of the form $v \otimes e_i$ for $i \in \{1,\dots, m\}$ and $v \in \C^n$. We have a natural unitary isomorphism 
\[
	\extp^* H = \extp^* (V_k \oplus V_k^\perp) \cong \extp^*V_k \otimes \extp^*V_k^\perp \cong \left(\extp^*\C^n\right)^{\otimes k} \otimes \extp^*V_k^\perp\ ,
\]
which gives rise to a representation $\pi_k \colon \Endo{\extp^*\C^n}^{\otimes k} \to \mathcal{B}(\extp^*H)$. Associativity of the above natural transformation ensures that this extends to a faithful representation $\pi \colon \MF \to \mathcal{B}(\extp^*H)$ of the colimit.

\begin{lemma}
The exterior algebra functor $F= \extp^*$ gives rise to a group homomorphism 
	\[	
		\varphi_{F} \colon \Udet{H} \to U(\MF)\ .
	\]
\end{lemma}

\begin{proof}
It suffices to see that $\pi(F(u)) \in \pi(\MF)$ for $u \in \Udet{H}$. If $u-1$ is a finite rank operator, then $H = V \oplus V^\perp$ for a finite-dimensional Hilbert subspace $V$ such that 
\[
	u = \begin{pmatrix}
		u_0 & 0 \\ 
		0 & 1
	\end{pmatrix}
\]
and $F(u) \in U(F(H))$ corresponds to $u_0 \otimes 1 \colon U(F(V) \otimes F(V^\perp))$ under the exponential transformation, which proves that $F(u) \in U(\MF)$ in this case. Now assume that $u \in U(H)$ is of the form $1 + T$ for a trace-class operator~$T$. Since $u$ is diagonalisable, so is $T$. Choose a Hilbert basis $(\xi_n)_{n \in \N_0}$ of eigenvectors for $T$ corresponding to eigenvalues $\lambda_n$. Define 
\[
	T_n = \sum_{k=0}^n \lambda_k\,\xi_k\,\lscal{}{\xi_k}{\,\cdot\,}\ .
\]
The $T_n$ converge to $T$ in trace norm. To prove the claim it suffices to see that $F(1 + T_n)$ converges to $F(1 + T)$ in norm, which will show that $F(1 + T) \in \pi(U(\MF))$. Let $V_n = \text{span}\{\xi_0, \dots, \xi_n\}$. Using the unitarity of the exponential transformation we obtain the estimates
\begin{gather*}
	\lVert F(1+T) - F(1+T_n) \rVert \leq 
	\left\lVert F(1 + \lambda_{n+1}) \otimes F(\left.(1+T)\right|_{V_{n+1}^\perp}) - 1 \otimes 1 \right\lVert \\
	\leq \lVert F(1 + \lambda_{n+1}) - 1 \rVert +  \left\lVert F(1 + \lambda_{n+2}) \otimes F(\left.(1+T)\right|_{V_{n+2}^\perp}) - 1 \otimes 1 \right\lVert \\
	\leq \sum_{k=n+1}^\infty \lVert F(1+\lambda_{k}) - 1 \rVert \leq \sum_{k=n+1}^\infty \lvert \lambda_k \rvert = \lVert T - T_n \rVert_1\ ,
\end{gather*} 
where the first inequality in the last line follows inductively. To get the last inequality we used that $F(\C) \cong \C \oplus \C$ and  
$F(1 + \lambda_{k}) - 1 =
\left(\begin{smallmatrix}
	0 & 0 \\
	0 & \lambda_{k}
\end{smallmatrix}\right)$. Hence, $F(1+T_n) \in U(\MF)$ converges in norm to $F(1+T) \in \mathcal{B}(F(H))$ proving the statement. 
\end{proof}

As a conclusion of the lemma, we may therefore replace $\det$ by our ``higher determinant'' $\varphi_{F}$ for $F = \extp^*$ to get a determinant bundle of the form 
\begin{equation} \label{eqn:UMF-Ures0}
	U(\MF) \to \cE \times_{\varphi_F} U(\MF) \to \Uidres{H}\ .
\end{equation}
and via pullback a corresponding bundle over the loop group $LG = LSU(n)$
\[
	U(\MF) \to LG_F \to LG\ .
\]
In the case of classical twists the bundle $LG_F$ came equipped with a group structure itself. This seems unlikely here. Nevertheless, the above higher determinant bundle suggests that one should look at representations of $\cE$ as in \eqref{eqn:Udet_ext} on Hilbert $\MF$-modules (or bimodules), which satisfy a suitable positive energy condition and such that $\Udet{H}$ acts through $\varphi_F$ (using the Hilbert module structure). In the case of $F = (\extp^{\rm top})^{\otimes (n+k)}$ this will then boil down to the Verlinde ring at level $k$. 

Concerning the multiplicative structure consider the homotopy given by 
\begin{gather*}
	H \colon SU(n) \times SU(n) \times [0,1] \to SU(2n) \\
	(u,v,t) \mapsto 
	\begin{pmatrix}
		u & 0 \\
		0 & 1 	
	\end{pmatrix}
	\begin{pmatrix}
		\cos(\tfrac{\pi}{2} t) & \sin(\tfrac{\pi}{2} t) \\
		-\sin(\tfrac{\pi}{2} t) & \cos(\tfrac{\pi}{2} t)
	\end{pmatrix}
	\begin{pmatrix}
		1 & 0 \\
		0 & v 	
	\end{pmatrix}
	\begin{pmatrix}
		\cos(\tfrac{\pi}{2} t) & -\sin(\tfrac{\pi}{2} t) \\
		\sin(\tfrac{\pi}{2} t) & \cos(\tfrac{\pi}{2} t)
	\end{pmatrix}
\end{gather*}
between the inclusion $(u,v) \mapsto (\begin{smallmatrix} u & 0 \\ 0 & v \end{smallmatrix})$ and $(u,v) \mapsto (\begin{smallmatrix} uv & 0 \\ 0 & 1 \end{smallmatrix})$. It is equivariant with respect to conjugation in the sense that 
\[
	H(gug^*, gvg^*, t) = 
	\begin{pmatrix}
		g & 0 \\
		0 & g
	\end{pmatrix}
	H(u,v,t)	
	\begin{pmatrix}
		g^* & 0 \\
		0 & g^*
	\end{pmatrix}\ .
\]
If we denote the Fell bundle over $SU(k)$ by $\cE_k$, the two projection maps $SU(n) \times SU(n) \to SU(n)$ by $p_i$ for $i \in \{1,2\}$ and the multiplication map by $\mu \colon SU(n) \times SU(n) \to SU(n)$, then we speculate that 
\begin{align*}
	H_0^*\cE_{2n} &\cong p_1^*\cE_n \otimes p_2^*\cE_n \ ,\\
	H_1^*\cE_{2n} &\cong \mu^*\cE_n \otimes \MF	\ .
\end{align*}
Hence, $H^*\cE_{2n}$ would give an equivariant homotopy between these two Fell bundles. Since after stabilisation by the compact operators the continuous $C(G)$-algebras $C^*(\cE_k)$ correspond to section algebras of locally trivial bundles, which are homotopy-invariant, one obtains an equivariant isomorphism of $C(G)$-algebras 
\[
	C^*(p_1^*\cE_n) \otimes C^*(p_2^*\cE_n) \otimes \bK \cong C^*(\mu^*\cE_n) \otimes \MF \otimes \bK\ .
\]
In the classical case the twist would give a cohomology class $\tau_k \in H^3_G(G,\Z)$. Multiplicativity of these twists is encapsulated by the equation
\[
	p_1^*\tau_k + p_2^*\tau_k = \mu^* \tau_k \quad \in H^3_G(G \times G, \Z)\ .
\]
The above isomorphism seems to be the correct generalisation of this statement and we will explore these ideas in future work.

There are more connections with conformal field theory as well: In \cite[Sec.~5.1.1]{paper:EvansPennig-Twists} we showed that for odd powers of the full exterior algebra twist on~$SU(2)$ the localisation is not necessary (in particular $F(\C)$ is already invertible in this case), and the $K$-groups give fusion rings related to the even part of $SU(2)$. These tadpole graphs also appear as fusion graphs for modules of $SU(3)$ as in e.g.\ \cite[page~12]{paper:EvansPugh}. This then raises the question of when localisation is really present for $SU(n)$ for $n > 2$ and whether the rings derived from the $K$-theory are fusion rings or fusion modules. Then there is the question of categorifying these (fusion) rings or modules to be fusion categories or their modules which is relevant for relating our rings to conformal field theory.

\bibliographystyle{plain}
\bibliography{SpectralSequence}

\end{document}